\documentclass[]{birkjour}
\usepackage{amsmath,amsthm,amsopn,amssymb,mathrsfs,times,newtxtext,newtxmath,upref}
\usepackage[mathscr]{eucal}

\usepackage{enumitem}

\setlist{label={$($\roman{enumi}\kern1pt$)$}}

\newtheorem{thm}{Theorem}[section]
\newtheorem{prop}[thm]{Proposition}
\newtheorem{cor}[thm]{Corollary}
\newtheorem*{cor*}{Corollary}
\newtheorem{lema}[thm]{Lemma}
\newtheorem*{lema*}{Lemma}

\numberwithin{equation}{section}
\theoremstyle{definition}

\newtheorem*{obs}{Remark}

\newcommand{\PI}[2]{\left\langle \,#1 , #2\, \right\rangle}

\newcommand{\St}{\mathcal{S}}
\newcommand{\HH}{\mathcal{H}}
\newcommand{\KK}{\mathcal{K}}
\newcommand{\M}{\mathcal{M}}
\newcommand{\LL}{\mathcal{L}}
\newcommand{\N}{\mathcal{N}}
\newcommand{\mc}[1]{\mathcal{#1}}

\newcommand{\ol}{\overline}
\newcommand{\clran}{\ol{\mathrm{ran}}\,}
\newcommand{\cldom}{\ol{\mathrm{dom}}\,}
\newcommand{\ra}{\rightarrow}

\newcommand{\matriz}[4]{\displaystyle\
	\left(
	\begin{array}{cc}
		{#1}&{#2}\\
		{#3}&{#4}
	\end{array}
	\right)}
\newcommand{\vect}[2]{\displaystyle\
	\left(
	\begin{array}{cc}
		{#1}&{#2}\\
	\end{array}
	\right)}

\newcommand{\vecd}[2]{\displaystyle\
	\left(
	\begin{array}{cc}
		{#1}\\{#2}
	\end{array}
	\right)} 

\newcommand{\vecdm}[2]{\displaystyle\
	\begin{array}{cc}
		{#1}\\{#2}
	\end{array}} 

\DeclareMathOperator{\ran}{ran}
\DeclareMathOperator{\dom}{dom}

\DeclareMathOperator{\mul}{mul}

\begin{document}

  \title[A matrix formula for Schur complements of nonnegative selfadjoint l.r.]{A matrix formula for Schur complements of nonnegative selfadjoint linear relations}

  \author[Contino]{Maximiliano Contino}
	
  \address{
    Instituto Argentino de Matem\'atica ``Alberto P. Calder\'on'' \\
    CONICET\\
    Saavedra 15, Piso 3\\
    (1083) Buenos Aires, Argentina \\[5pt]
    Facultad de Ingenier\'{\i}a, Universidad de Buenos Aires\\
    Paseo Col\'on 850 \\
    (1063) Buenos Aires, Argentina}
    \email{mcontino@fi.uba.ar}

  \author[Maestripieri]{Alejandra~Maestripieri}
	
  \address{
    Instituto Argentino de Matem\'atica ``Alberto P. Calder\'on'' \\
    CONICET\\
    Saavedra 15, Piso 3\\
    (1083) Buenos Aires, Argentina \\[5pt]
    Facultad de Ingenier\'{\i}a, Universidad de Buenos Aires\\
    Paseo Col\'on 850 \\
    (1063) Buenos Aires, Argentina}
    \email{amaestri@fi.uba.ar}
  
  \author[Marcantognini]{Stefania Marcantognini}
	
  \address{
    Instituto Argentino de Matem\'atica ``Alberto P. Calder\'on'' \\
    CONICET\\
    Saavedra 15, Piso 3\\
    (1083) Buenos Aires, Argentina \\[5pt]
    Universidad Nacional de General Sarmiento -- Instituto de Ciencias
    \\ Juan Mar\'ia Gutierrez 
    \\ (1613) Los Polvorines, Pcia. de
    Buenos Aires, Argentina}
    \email{smarcantognini@ungs.edu.ar}
	


\date{\today}

\begin{abstract} If a nonnegative selfadjoint linear relation $A$ in a Hilbert space and a closed subspace $\St$ are assumed to satisfy that the domain of $A$ is invariant under the orthogonal projector onto $\St,$ then $A$ admits a particular  matrix representation with respect to the decomposition $\St \oplus \St^\perp$. This matrix representation of $A$ is used to give explicit formulae for the Schur complement of $A$ on $\St$ as well as the $\St-$compression of $A$.
\end{abstract}
    
 \maketitle
	
\section{Introduction}
\label{sec:introduction}
Given a nonnegative selfadjoint linear relation $A$ in a Hilbert space $\HH$ and a closed subspace $\St$ of $\HH,$
it is not always the case that  $A$ admits a $2\times 2$ block matrix representation with respect to the decomposition $\St \oplus \St^\perp$. On the other hand, if it does, the matrix representation need not be unique. Results on this subject can be found in \cite{Engel, Moller, Hassi2, Chen, HassiMR}.
Under the hypothesis that ${\rm{dom}}(A)$ (the domain of $A$) is an invariant subspace for the orthogonal projection onto $\St,$ $P_\St$ (that is $\mc{D}_1:=P_{\St}(\dom(A)) \subseteq \dom(A)$), we show that $A$ can be represented by a  $2\times 2$ block matrix $( \begin{smallmatrix} a&b\\c&d\end{smallmatrix} )$ where $a, b, c$ and $d$ are linear relations. Furthermore, $A$ admits a specific representation $( \begin{smallmatrix} a&b\\c&d\end{smallmatrix} )$ similar to the one for bounded operators (cf. \cite{Shorted2}, \cite[Lema A.1]{Dritschel}), in the sense that $a$ and $d$  in this decomposition are nonnegative selfadjoint linear relations and there exists a contraction $g: \St^\perp \to \St$ such that $b=a^{1/2}gd^{1/2}|_{P_{\St^\perp}({\rm{dom}}(A))}$ and  $c=d^{1/2}g^*a^{1/2}|_{P_\St({\rm{dom}}(A))}$.  

In \cite{Arlinskii}, Arlinski\u{\i } proves that for $\leq$ the forms order \cite{Kato,Behrndt},  the maximum of the following set of nonnegative selfadjoint linear relations,
$$\{ X : 0 \leq X \leq A, \ \ran(X) \subseteq \St^{\perp}\}$$ always exists and he defines the Schur complement of the relation $A$ with respect to $\St,$ $A_{/ \St},$ as this maximum. 
Under the invariance condition mentioned above,  we give a matrix formula for  $A_{/\St}$ in terms of the matrix coefficients of $A;$ namely, $$A_{/\St}=\matriz{0}{0}{0}{T^*T}$$ with
$T:=\ol{D_gd^{1/2}|_{P_{\St^\perp}({\rm{dom}}(A))}},$ where $D_g:=(1-g^*g)^{1/2}$ is the defect operator associated to the matrix representation of $A.$ We also give an alternate proof of the existence of the Schur complement. This formula is an extension of the well known formula by Anderson and Trapp for bounded operators \cite{Shorted2}. We also define the $\St$-compression $A_{\St}$  of $A$. 
If we assume further that ${\rm{dom}}(A^{1/2})$ is an invariant subspace of the orthogonal projection $P_\LL$, 
where 
$\LL:=\ol{A^{1/2}(\mc{D}_1)} \cap \cldom(A),$ then we obtain Pekarev-type formulae for $A/_\St$ and $A_{\St}$ \cite{Pekarev}, and we show that $A = A_{\St}+ A/_\St$. 

The paper is organized as follows. In Section 2 we outline some background material, primarly on linear relations. Section 3 is devoted to the problem of representing a selfadjoint linear relation $A$ as a $2\times 2$ relation matrix $( \begin{smallmatrix} a&b\\c&d\end{smallmatrix} )$ with respect to the decomposition $\St \oplus \St^\perp$. 
In Proposition \ref{propMDLR},  we prove that the relation $A$ admits a $2 \times 2$ block matrix representation with respect to $\St \oplus \St^{\perp}$ if and only if its operator part $A_0$ admits  a block matrix representation with respect to $\ol{\mc{D}_1}$  plus the extra condition $\St \ominus \mc{D}_1 \subseteq \mul(A)$ (the multivalued part of $A$). The main result of this section is Theorem \ref{thmRMRL}, where this matrix representation of $A$ is fully described when $A$ is nonnegative.
In  Section 4,  we  again use the matrix representation of the nonnegative selfadjoint  linear relation $A$ to derive formulae for the Schur complement and compression of $A.$

\section{Preliminaries}
\label{sec:preliminaries}

Throughout, all spaces are complex and separable Hilbert spaces.  As usual, the direct sum of two subspaces $\M$ and $\N$ of a Hilbert space $\HH$ is indicated by $\M \dotplus \N$ and the
orthogonal direct sum by $\M \oplus \N.$   The orthogonal complement of a subspace $\mc{M} \subseteq \HH$ is written as $\mc{M}^\perp,$ or $\HH \ominus \M$ interchangeably. The symbol $P_{\M}$ denotes the orthogonal projection with range $\M$.

The space of everywhere defined bounded linear operators from $\HH$ to
$\KK$ is written as $L(\HH, \KK)$, or $L(\HH)$ when $\HH = \KK.$
The identity operator on $\HH$ is written as $1$, or $1_{\HH}$ if it
is necessary to disambiguate.

The notion of Schur complement (or shorted operator) of $A$ to $\St$ for a nonnegative selfadjoint operator $A \in L(\HH)$ and $\St \subseteq \HH$ a closed subspace, was introduced by M.G.~Krein \cite{Krein}. When $\leq$ is the usual order in $L(\HH)$, he proved that the set $\{ X \in L(\HH): \ 0\leq X\leq A \mbox{ and } \ran(X)\subseteq \St^{\perp}\}$ has a maximum element, which he defined as the {{Schur complement}} $A_{/ \St}$ of $A$ to $\St.$ This notion was later rediscovered by Anderson and Trapp \cite{Shorted2}. If $A$ is represented as the $2\times 2$ block matrix $( \begin{smallmatrix} a&b\\b^*&d\end{smallmatrix} )$ with respect to the decomposition of $\HH = \St \oplus \St^{\perp},$ they established the formula 
$$A_{/ \St}= \begin{pmatrix} 0 & 0\\ 0& d - y^*y\end{pmatrix}$$ where $y$ is the unique solution of the equation $b = a^{1/2} x$ such that the range inclusion $\ran(y) \subseteq \clran(a)$ holds.

\medskip
Although familiarity with the theory of linear relations is presumed, some background material from \cite{Hassi} is summarized below.

A linear relation (l.r.) from a Hilbert space $\HH$ to a Hilbert space $\KK$ is a linear subspace $T$ of the cartesian product $\HH \times \KK.$ 
The domain,  range, null space or kernel and multivalued part of $T$ is denoted by $\dom(T),$  $\ran(T),$ $\ker(T)$ and  $\mul(T),$ respectively. When $\mul(T)=\{0\},$ $T$ is an operator; in this case, the operator $T$ is uniquely determinated by $Tx=y$ for $(x,y) \in T.$ 

The sum of two linear relations $T$ and $S$ from $\HH$ to $\KK$ is the linear relation defined by
 $$T+S:=\{(x,y+z): (x,y ) \in T \mbox{ and } (x,z) \in S\}.$$
 The componentwise sum is the linear relation defined by
 $$T \ \hat{+}  \ S:=\{(x_1+x_2,y+z): (x_1,y ) \in T \mbox{ and } (x_2,z) \in S\}.$$ The componentwise sum of $T$ and $S$ with $T \perp S$ is denoted by $T\ \hat{\oplus} \ S.$ 
Let $T$ be a linear relation from  $\HH$ to a Hilbert space $\mc{E}$ and let $S$ be a linear relation from $\mc{E}$ to  $\mc{K}$ then the product $ST$ is a linear relation from $\HH$ to $\KK$ defined by
 $$ST:=\{(x,y): (x,z) \in T \mbox{ and } (z,y) \in S \mbox{ for some } z \in \mc{E}\}.$$
 If $T \in L(\HH, \mc{E})$ then $(x,y) \in ST$ if and only if $(Tx,y) \in S.$
 
 The closure of a linear relation from $\HH$ to $\KK$ is the closure of the linear subspace in $\HH \times \KK,$ when the product is provided with the product topology. The closure of an operator need not be an operator; if it is then one speaks of a \emph{closable operator}. The relation $T$ is called \emph{closed} when it is closed as a subspace of $\HH \times \KK.$ 
 The \emph{adjoint} relation from $\KK$ to $\HH$ is defined by $$T^*:=JT^{\perp}=(JT)^{\perp},$$ where $J(x,y)=(y,-x).$ The adjoint 
 is automatically a closed linear relation and, if $\ol{T}$  denotes the closure of $T,$ then  $\ol{T}=T^{**}:=(T^*)^*.$ By definition, it is immediate that $\ol{T}^*=T^*.$  Clearly,
 $$T^*=\{(x,y) \in \KK \times \HH: \PI{g}{x}=\PI{f}{y} \mbox{ for all } (f,g) \in T \}.$$ Hence $\mul(T^*)=\dom(T)^{\perp}$ and $\ker(T^*) =\ran(T)^{\perp}.$ Then, if $T$ is closed both $\ker(T)$ and $\mul(T)$ are closed subspaces.
 
Let $T$ be a linear relation from  $\HH$ to a Hilbert space $\mc{E}$ and let $S$ be a linear relation from $\mc{E}$ to $\mc{K}$ then
\begin{equation} \label{product}
T^*S^* \subset (ST)^*	
\end{equation}
and there is equality in \eqref{product} if $S \in L(\mc{E}, \mc{K})$. If $T$ and $S$ are linear relations from $\HH$ to $\KK$ then 
 \begin{equation} \label{sum}
 T^*+S^* \subset (T+S)^*	
 \end{equation}
 and there is equality in \eqref{sum} if $S \in L(\mc{H}, \mc{K}).$
 
 Let $T$ be a (not necessarily closed) linear relation in $\HH.$ Define $T_0:=T \cap (\cldom(T) \times \cldom(T^*))$ and $T_{\mul}:=\{0\} \times \mul(T).$ Then $T_0$ is a closable operator from $\cldom(T)$ to $\cldom(T^*)$ \cite{Hassi2}. 
 
 \begin{thm}[{\cite[Theorem 3.9]{Hassi2}}]  \label{TdecompHassi} Let $T$ be a (not necessarily closed) linear relation in $\HH.$ If there exists a linear relation $B$ in $\HH$ such that
\begin{equation} \label{decom}
T=B \  \hat{+} \ T_{\mul}, \  \ \ \ran(B) \subseteq \cldom(T^*),
\end{equation}
 then the sum in \eqref{decom} is direct and $B$ is a closable operator which coincides with
 $T_0.$ In particular, the decomposition of $T$ in \eqref{decom} is unique.
 \end{thm}
 
Hence if $T$ admits a componentwise sum decomposition of the form  \eqref{decom} then, since $\cldom(T^*) =\mul(\ol{T})^{\perp} \subseteq \mul(T)^{\perp},$ it follows that
 \begin{equation} \label{Tdecom}
T=T_0 \ \hat{\oplus} \ T_{\mul}.
\end{equation} 
We say that $T$ is \emph{decomposable} if $T$  admits the componentwise sum decomposition \eqref{decom}, or equivalently, \eqref{Tdecom}.
 
%
%
In particular, if $T$ is a closed linear relation in $\HH$  then $\mul(T)=\dom(T^*)^{\perp}$ and $T$ is decomposable and \eqref{Tdecom} is valid.
In this case, $T_0$ is a closed operator from $\cldom(T)$ to $\cldom(T^*)$ and $T_{\mul}$ is a closed linear relation. 
Also, $\dom(T_0)=\dom(T)$ and $\ran(T_0) \subseteq  \cldom(T^*).$ The \emph{operator part} $T_0$ is densely defined in $\cldom(T)$ and maps into $\cldom(T^*)$. The operator parts $T_0$ and $(T^*)_0$ are connected by
 \begin{equation} \label{T0}
 	(T_0)^{\times}=(T^*)_0,
 \end{equation}
 where $(T_0)^{\times}$ denotes the adjoint of $T_0$ when viewed as an operator from $\cldom(T)$ to $\cldom(T^*).$
 
 A linear relation $T$  in $\HH$ is \emph{symmetric} if $T \subset T^*,$   \emph{selfadjoint} if $T=T^*$ and \emph{nonnegative} if $\PI{y}{x} \geq 0$ for all $(x,y) \in T.$  If $T$ is a nonnegative selfadjoint linear relation we write $T \geq 0.$

 \begin{lema} \label{lemaTsa}  Let $T$ be a closed linear relation in $\HH$ and suppose that $T =T_0 \ \hat{\oplus} \ T_{\mul}$ as in \eqref{Tdecom}.
Then $T$ is selfadjoint if and only if $\cldom(T^*)=\cldom(T)$ and $T_0$ is a selfadjoint operator in $\cldom(T).$
 \end{lema}
\begin{proof} If $T=T^*$ then clearly $\cldom(T^*)=\cldom(T)$ and, by \eqref{T0}, $(T_0)^{\times}=(T^*)_0=T_0$ \cite{Hassi}. Conversely, suppose that $\cldom(T^*)=\cldom(T)$ and $T_0$ is a selfadjoint operator in $\cldom(T).$ Then $\mul(T)=\dom(T^*)^{\perp}=\dom(T)^{\perp}=\mul(T^*)$ and,  by \eqref{T0}, $(T^*)_0=(T_0)^{\times}=T_0.$ So that
$$T^*=(T^*)_0 \ \hat{\oplus} \ (\{0\} \times \mul(T^*))=T_0 \ \hat{\oplus} \ (\{0\} \times \mul(T))=T.$$
\end{proof}



Next  a well-known result due to von Neumann (see \cite[Proposition 3.18]{Schmudgen}) is extended to closed linear relations:

\begin{thm}[{\cite[Lemma 2.4]{Hassi}}] \label{thmVN} Let $T$ be a closed linear relation in $\HH.$ Then  $T^{*}T$ is a nonnegative selfadjoint linear relation in $\HH.$ Furthermore, 
\begin{equation} \label{TTLR1}
T^*T =T^*T_0={T_0}^*T_0,
\end{equation}
where $T_0$ is the operator part of $T.$
In particular
\begin{equation} \label{TTLR2}
\ker(T^*T)=\ker(T)=\ker(T_0) \mbox{ and } \mul(T^*T)=\mul(T^*)=\mul({T_0}^*).
\end{equation}
Also, the operator part of $T^*T$ is 
\begin{equation} \label{TTLR3}
(T^*T)_0=(T^*)_0T_0=(T_0)^{\times}T_0.
\end{equation}	
\end{thm}

%
%

Let $T \geq 0$ be  a linear relation in $\HH.$ Since $T$ is selfadjoint (and therefore closed), $\mul(T)=\dom(T)^{\perp}.$  Hence
$\HH= \cldom(T) \oplus \mul(T).$ In this case  $T$ can be written as
$T=T_0 \ \hat{\oplus} \ T_{\mul}$ where, by Lemma \ref{lemaTsa}, $T_0$ is a nonnegative selfadjoint operator in $\cldom(T).$ 
For $T \geq 0,$ the (unique) nonnegative selfadjoint \emph{square root} of $T$ is defined by
$$T^{1/2}:=T_0^{1/2}\ \hat{\oplus} \ (\{0\} \times \mul(T)),$$
where $T_0^{1/2}$ is the square root of $T_0$  \cite{Bernau}. 
Then, $\mul(T^{1/2})=\mul(T),$  $T_0^{1/2}=(T^{1/2})_0$ and $\cldom(T)=\cldom(T^{1/2})$ \cite[Lemma 2.5]{Hassi}.
 Also, by \eqref{TTLR2},
 \begin{equation} \label{eqker}
 	\ker(T)= \ker(T^{1/2})=	\ker(T_0).
 \end{equation}

There is a natural ordering for nonnegative selfadjoint relations in $\HH.$  For two nonnegative selfadjoint relations $A$ and $B,$ we write $A \leq B$ if
\begin{equation}
\dom(B_0^{1/2}) \subseteq\dom(A_0)^{1/2}  \mbox{ and } \Vert A_0^{1/2} u \Vert \leq  \Vert B_0^{1/2} u \Vert, \mbox{ for all } u \in \dom(B_0^{1/2}). \!
\end{equation}

The following is a result given in \cite[Theorem 3.4]{Hassi}; we include its proof for the sake of completeness. 
\begin{lema} \label{lemalr} Let $A, B$ be nonnegative selfadjoint linear relations such that
	$A \leq B.$ Then, there exists a contraction $W \in L(\cldom(B), \cldom(A))$ such that
\begin{equation} \label{contencion}
WB_0^{1/2} \subset A_0^{1/2}
\end{equation}
where $A_0$ and $B_0$ are the operator parts of $A$ and $B,$ respectively.
\end{lema}
\begin{proof} Since $A \leq B,$ $\dom(B_0^{1/2})\subseteq \dom(A_0^{1/2})$ and 
\begin{equation} \label{eqLR1}
\Vert A_0^{1/2} u \Vert \leq \Vert B_0^{1/2} u \Vert,
\end{equation}
for every $u \in \dom(B_0^{1/2}).$
Define the linear relation
$$W:=\{(B_0^{1/2}h,A_0^{1/2}h): h \in \dom(B_0^{1/2})\}.$$ If $(x,y) \in W$ then $(x,y)=(B_0^{1/2}h,A_0^{1/2}h)$ for some $h \in \dom(B_0^{1/2}).$ Then, by \eqref{eqLR1},
$$\Vert y \Vert = \Vert A_0^{1/2}h \Vert \leq \Vert B_0^{1/2}h\Vert=\Vert x \Vert.$$
So that $W$ is a contraction  from $\ran(B_0^{1/2})$ to $\ran(A_0^{1/2}).$ Then $W$ has a unique extension named again $W$ from $\clran(B_0^{1/2}) \subseteq \cldom(B)$ to $\clran(A_0^{1/2}) \subseteq \cldom(A).$ Defining $W$ as zero in $\cldom(B)  \ominus \ran(B_0^{1/2}),$ the result follows. 
\end{proof}


 If $T$ is a linear relation in $\HH \times \KK$ and $\St$ is a subspace of $\dom(T)$ then 
 $$T|_{\St}:=\{(x,y) \in T : x \in \St\} \mbox{ and } T(\mc{S}):=\{y : (x,y) \in T \mbox{ for some } x \in \mc{S} \}.$$ 
 
 A linear subspace $\mc{D}$ of $\dom(T)$ is a \emph{core} of $T$ if the set $T|_{\mc{D}}$ is dense in $T,$ in which case  $\ol{T({\mc{D}})}=\clran{T}.$ If $T$ admits the sum decomposition $T=T_0 \ \hat{\oplus} \ T_{\mul}$ as in \eqref{Tdecom} and $\mc{D}$ is a core of $T_0$ then $\mc{D}$ is a core of $T.$
 If $T$ is a selfadjoint linear relation in $\HH$ and $\mc{D}$ is a core of $T$ then  $(T|_{\mc{D}})^*=T.$
 
\section{Matrix decomposition  of nonnegative selfadjoint relations}
Let $\St$ be a closed subspace of $\HH$ and let $a \subseteq \St \times\St,$ $b  \subseteq \St^{\perp} \times\St,$ $c  \subseteq \St \times\St^{\perp}$ and $d  \subseteq \St^{\perp} \times \St^{\perp}$ be linear relations. In \cite[Definition 5.1]{HassiMR},  the  linear relation in $\HH \times \HH$ \emph{generated} by the blocks $a,$  $b,$ $c$ and $d$ is defined as
$$\matriz{a}{b}{c}{d}:=\left\{ \left(\vecd{x_1}{x_2}, \vecd{w_1+z_1}{w_2+z_2}\right): \vecdm{(x_1,w_1) \in a, (x_2,z_1)\in b}{(x_1,w_2) \in c, (x_2,z_2)\in d}\right\}.$$

On the other hand, given a linear relation $A$ in $\HH$ and $\St$ a closed subspace of $\HH,$ we say that $A$ \emph{admits a $2 \times 2$ block matrix representation with respect to $\St \oplus \St^{\perp}$} if there exist blocks  $a \subseteq \St \times\St,$ $b \subseteq \St^{\perp} \times\St,$ $c  \subseteq \St \times\St^{\perp}$ and $d \subseteq \St^{\perp} \times \St^{\perp}$ such that 
$A=\matriz{a}{b}{c}{d}.$ 
In this case, it is easy to check that:
\begin{enumerate}
\item [1.] $\dom(a)  \cap \dom(c) = \St \cap \dom(A)$ and $\dom(b)  \cap \dom(d) = \St^{\perp} \cap \dom(A).$ 
\item  [2.] $\mul(a)  + \mul(b) = \St \cap \mul(A)$ and $\mul(c)  + \mul(d) = \St^{\perp} \cap \mul(A).$ 
\end{enumerate}

\begin{lema} \label{lemasubs} Let $\M$ and $\St$ be subspaces of $\HH$ with $\St$ closed.
Then the following are equivalent:
\begin{enumerate}
	\item $P_{\St} (\M) \subseteq \M;$
	\item  $\M=\St \cap \M \oplus  \St^{\perp}\cap \M;$
	\item $P_{\St}(\M)=\St \cap \M.$ 
\end{enumerate}
\end{lema}

\begin{thm}[\normalfont{cf. \cite[Theorem 5.1]{HassiMR}}] \label{theoMLR} Let $A$ be a linear relation in $\HH$ and let $\St$ be a closed subspace of $\HH.$ Then the following are equivalent:
	\begin{enumerate}
		\item $A$ admits a $2 \times 2$ block matrix representation with respect to $\St\oplus \St^{\perp};$
		\item  $P_{\St} (\dom(A)) \subseteq \dom(A)$ and $P_{\St} (\mul(A)) \subseteq \mul(A);$ 
		\item $A$ admits a representation as 
		\begin{equation} \label{Acanon}
		A=\matriz{a}{b}{c}{d},
		\end{equation} where $a:=P_{\St}A|_{\St},$  $b:=P_{\St}A|_{\St^{\perp}},$ $ c:=P_{\St^{\perp}}A|_{\St}$ and  $d:=P_{\St^{\perp}}A|_{\St^{\perp}}.$
	\end{enumerate}
\end{thm}
 

\begin{lema} \label{lemmamul}
Let $A$ be a selfadjoint  linear relation in $\HH$ and let $\St$ be a closed subspace of $\HH.$
If $P_{\St} (\dom(A)) \subseteq \dom(A)$ then $P_{\St} (\mul(A)) \subseteq \mul (A).$ 

\end{lema}
\begin{proof} Since $A$ is selfadjoint,  $\mul(A)=\dom(A)^{\perp}.$  Let  $y \in \mul(A).$  Then, for all $h \in \dom(A)$
$$\PI{P_{\St} y}{h}=\PI{y}{P_{\St} h}=0,$$ because $P_{\St} h \in \dom(A).$
Therefore $P_{\St} y \in \dom(A)^{\perp}=\mul(A).$

\end{proof}

Let $A$ be a selfadjoint linear relation in $\HH$ and let $\St$ be a closed subspace of $\HH.$ Define
\begin{align}
\label{D1} \mc{D}_1&:=\St \cap \dom(A), \ \mc{D}_2:=\St^{\perp} \cap \dom(A), \\
\label{M1} \M_1&:=\St\cap \mul(A) \mbox{ and } \M_2:=\St^{\perp}\cap \mul(A).
\end{align} 
If $P_{\St} (\dom(A)) \subseteq \dom(A)$ then, by Lemmas \ref{lemasubs} and  \ref{lemmamul}, 
\begin{equation} \label{dom1}
\dom(A)=\mc{D}_1 \oplus \mc{D}_2 \mbox{ and } 	\mul(A)=\M_1 \oplus \M_2,
\end{equation}
and $A$ admits a $2 \times 2$ block matrix representation with respect to $\St\oplus \St^{\perp}.$

Define $\N_i:=\ol{\mc{D}_i}, \mbox{ for } i=1,2.$
Clearly, $\cldom(A)=\N_1 \oplus \N_2.$

\begin{lema} \label{lemaRMN2} Let $A$ be a selfadjoint  linear relation in $\HH$ and let $\St$ be a closed subspace of $\HH.$ Then, the following are equivalent:
\begin{enumerate}
\item $P_{\St} (\dom(A)) \subseteq \dom(A);$
\item $P_{\N_1}(\dom(A)) \subseteq \dom(A)$ and  $\St=\N_1 \oplus \M_1;$
\item  $P_{\N_2}(\dom(A)) \subseteq \dom(A)$ and  $\St^{\perp}=\N_2 \oplus \M_2.$
\end{enumerate}
In this case, $\N_1=\St\cap \cldom(A)$  and $\N_2=\St^{\perp}\cap \cldom(A).$
\end{lema}

\begin{proof}  $(\mathit{i}\kern0.5pt) \Leftrightarrow (\mathit{ii}\kern0.5pt)$: If $P_{\St} (\dom(A)) \subseteq \dom(A)$ then by \eqref{dom1}, $\cldom(A)=\N_1 \oplus \N_2,$ 
$\mc{D}_1=\N_1 \cap \dom(A)$ and $\mc{D}_2=\N_2 \cap \dom(A).$ 
Therefore 
\begin{equation}\label{Ndom}
\dom(A)=\N_1 \cap \dom(A) \oplus \N_2 \cap \dom(A).
\end{equation}
Hence  $P_{\N_1}(\dom(A))= \mc{D}_1\subseteq \dom(A).$ 

Also $\cldom(A) \subseteq (\St \ominus \N_1)^{\perp}$ or, equivalently,  $\St \ominus \N_1 \subseteq \mul(A).$ In fact, 
$(\St \ominus \N_1)^{\perp} = \St^{\perp} \oplus \N_1 \supseteq \N_2 \oplus \N_1 = \cldom(A).$
Hence
$$\St = \N_1 \oplus (\St \ominus \N_1) \subseteq \N_1 \oplus (\St \cap \mul(A))=\N_1 \oplus \M_1 \subseteq \St.$$
Conversely, suppose that $P_{\N_1}(\dom(A)) \subseteq \dom(A)$ and  $\St=\N_1 \oplus \M_1.$ Then $P_{\St}=P_{\N_1} + P_{\M_1}.$ Since $\dom(A) \subseteq \mul(A)^{\perp} \subseteq \M_1^{\perp},$ it follows that
$$P_{\St} (\dom(A))=(P_{\N_1} + P_{\M_1})(\dom(A))= P_{\N_1}(\dom(A))\subseteq \dom(A).$$

$(\mathit{i}\kern0.5pt) \Leftrightarrow (\mathit{iii}\kern0.5pt)$: It follows as $(\mathit{i}\kern0.5pt) \Leftrightarrow (\mathit{ii}\kern0.5pt)$ using that $P_{\St^{\perp}}(\dom(A)) \subseteq \dom(A).$

In this case,  $\N_1=\St\cap \cldom(A).$ The inclusion $\N_1 =\ol{\St \cap \dom(A)} \subseteq  \St\cap \cldom(A)$ always holds. Conversely, if $x \in  \St\cap \cldom(A)$  write $x=x_1+x_2,$ with $x_1 \in \N_1$ and $x_2 \in  \N_2.$ Then $x_2=x-x_1 \in \St \cap \St^{\perp}.$ So that $x_2 =0.$ Likewise, $\N_2=\St^{\perp}\cap \cldom(A).$
\end{proof}

Now, suppose that the selfadjoint linear relation $A$ is written as
\begin{equation} \label{A0}
A=A_0 \ \hat{\oplus} \ A_{\mul},
\end{equation} 
where $A_0$ is the selfadjoint operator part of $A$ in $\cldom(A).$ 

\begin{prop} \label{propMDLR} Let  $A$ be a selfadjoint linear relation in $\HH,$ let $\St$ be a closed subspace of $\HH$ and suppose that $A$ is written as in \eqref{A0}. Then $A$ admits a $2 \times 2$ block matrix representation with respect to $\St\oplus \St^{\perp}$ if and only if $A_0$ admits a $2 \times 2$ block matrix representation with respect to $\N_1 \oplus \N_2$ and $\St=\N_1 \oplus \M_1,$ where $\N_1=\ol{\mc{D}_1},$ $\N_2=\ol{\mc{D}_2},$ and $\mc{D}_1, \mc{D}_2 $ and $\mc{M}_1$ are defined as in \eqref{D1} and \eqref{M1}.
\end{prop}
\begin{proof} If $A$ admits a $2 \times 2$ block matrix representation with respect to $\St\oplus \St^{\perp},$ by Theorem \ref{theoMLR},
	$P_{\St} (\dom(A)) \subseteq \dom(A).$ Then, by Lemma \ref{lemaRMN2}, equation \eqref{Ndom} follows and $P_{\N_1 {\mathbin{/\mkern-3mu/}} \N_2}(\dom(A_0)) \subseteq \dom(A_0),$ where $P_{\N_1 {\mathbin{/\mkern-3mu/}} \N_2}$ is the orthogonal projection onto $\N_1$ in $L(\cldom(A_0)).$ Therefore, by Theorem \ref{theoMLR} the linear operator $A_0$ admits a $2 \times 2$ block matrix representation (in $\cldom(A_0)$) with respect to $\N_1 \oplus \N_2$ and, by Lemma \ref{lemaRMN2},  $\St=\N_1 \oplus \M_1.$ Conversely, if the linear operator $A_0$ admits a $2 \times 2$ block matrix representation with respect to $\N_1 \oplus \N_2,$ by Theorem \ref{theoMLR}, $P_{\N_1 {\mathbin{/\mkern-3mu/}} \N_2}(\dom(A_0)) \subseteq \dom(A_0).$ So that, by Lemma \ref{lemasubs}, equation \eqref{Ndom} follows. Then, $P_{\N_1} (\dom(A)) \subseteq \dom(A)$ and, since  $\St=\N_1 \oplus \M_1,$ by Lemma \ref{lemaRMN2}, $P_{\St} (\dom(A)) \subseteq \dom(A).$ Hence, by Theorem \ref{theoMLR}, $A$ admits a $2 \times 2$ block matrix representation with respect to $\St\oplus \St^{\perp}.$ 	
\end{proof}

\begin{cor} \label{corMDLR} Let  $A$ be a selfadjoint linear relation in $\HH,$ let $\St$ be a closed subspace of $\HH$ such that $P_\St(\dom(A)) \subseteq \dom(A)$ and suppose that $A$ is written as in \eqref{A0}. 
	
If $A_0$ admits the representation with respect to $\N_1 \oplus \N_2,$ $A_0=\matriz{a_0}{b_0}{c_0}{d_0},$ then $A$ admits the representation with respect to $\St\oplus \St^{\perp},$
$$A=\matriz{a_0 \ \hat{\oplus } \ (\{0\} \times \M_1')}{b_0 \ \hat{\oplus } \ (\{0\} \times \M_1'')}{c_0 \ \hat{\oplus } \ (\{0\} \times \M_2')}{d_0 \ \hat{\oplus } \ (\{0\} \times \M_2'')},$$ where $\M_1', \M_1''$ are subspaces of $\St$ and  $\M_2', \M_2''$ are subspaces of  $\St^{\perp}$ such that $\M_1'+\M_1''=\M_1$ and $\M_2'+\M_2''=\M_2.$

Conversely, if $A$ admits the representation with respect to $\St\oplus \St^{\perp},$
$
A=\matriz{a}{b}{c}{d},
$ then $A_0$ admits the representation with respect to $\N_1 \oplus \N_2,$
$$
A_0=\matriz{P_{\N_1}a}{P_{\N_1}b}{P_{\N_2}c}{P_{\N_2}d}.
$$
\end{cor}

\begin{proof} 
Suppose that $A_0$ admits the representation with respect to $\N_1 \oplus \N_2$
$$A_0=\matriz{a_0}{b_0}{c_0}{d_0}.$$

Set $a:=a_0 \ \hat{\oplus } \ \{ \{0\} \times \M_1' \},$ $b:=b_0 \ \hat{\oplus } \ (\{0\} \times \M_1''),$ $c:=c_0 \ \hat{\oplus } \ (\{0\} \times \M_2'),$ $d:=d_0 \ \hat{\oplus } \ (\{0\} \times \M_2'').$  Since $\M_1', \M_1'' \subseteq \M_1,$ $\M_2', \M_2'' \subseteq \M_2,$ $\St=\N_1\oplus \M_1$ and $\St^{\perp}=\N_2 \oplus \M_2,$ it is clear that $a \subseteq \St \times\St,$ $b  \subseteq \St^{\perp} \times\St,$ $c  \subseteq \St \times\St^{\perp}$ and $d \subseteq \St^{\perp} \times \St^{\perp}.$ 
Also,
\begin{align*}
\dom \matriz{a}{b}{c}{d}&=\dom(a) \cap \dom(c) \oplus \dom(b) \cap \dom(d)\\
&=\dom(a_0) \cap \dom(c_0) \oplus \dom(b_0) \cap \dom(d_0)\\
&=\N_1\cap \dom(A) \oplus \N_2 \cap \dom(A)\\
&=\mc{D}_1 \oplus \mc{D}_2 =\dom(A), 
\end{align*}
and
\begin{align*}
\mul\matriz{a}{b}{c}{d}&= \mul(a) + \mul(b) \oplus \mul(c) + \mul(d)\\
&=\M_1'+\M_1'' \oplus \M_2'+\M_2''=\M_1 \oplus \M_2= \mul(A).
\end{align*}

Let $(x,y) \in A=A_0 \ \hat{\oplus} \  (\{0\} \times \mul(A)).$ Then there exists $m \in \mul(A)$ such that $(x,y)=(x,A_0x)+(0,m).$
Then $x=x_1+x_2$ for some $x_1 \in \mc{D}_1$ and $x_2 \in \mc{D}_2$ and $m=m_1+m_2$ for some $m_1 \in \M_1$ and $m_2 \in \M_2.$  Since $m_1 \in \M_1$ and $m_2 \in \M_2,$ there exist $m_1' \in \M_1',$ $m_1'' \in \M_1'',$ $m_2' \in \M_2'$ and $m_2'' \in \M_2''$ such that
$m_1=m_1'+m_1''$ and $m_2=m_2'+m_2''.$ Then
\begin{align*}
	(x,y)&=(x,A_0x)+(0,m)=\left(\vecd{x_1}{x_2}, \matriz{a_0}{b_0}{c_0}{d_0} \vecd{x_1}{x_2}\right)+\left(0, \vecd{m_1}{m_2}\right)\\
	&=\left(\vecd{x_1}{x_2}, \vecd{a_0x_1+b_0x_2+m_1}{c_0x_1+d_0x_2+m_2}\right)\\
	&=\left(\vecd{x_1}{x_2}, \vecd{a_0x_1+b_0x_2+m_1'+m_1''}{c_0x_1+d_0x_2+m_2'+m_2''}\right).
\end{align*}
Now, since $(x_1, a_0x_1+m_1') =(x_1,a_0x_1)+(0,m_1')\in a,$ $(x_2,b_0x_2+m_1'') = (x_2,b_0x_2)+(0,m_1'')\in b,$ $(x_1,c_0x_1+m_2')=(x_1,c_0x_1)+(0,m_2') \in c$ and $(x_2,d_0x_2+m_2'')=(x_2,d_0x_2)+(0,m_2'') \in d,$ it follows that
$(x,y)\in \matriz{a}{b}{c}{d}.$
Hence, $A \subset \matriz{a}{b}{c}{d}$ and, since $\dom(A)=\dom\matriz{a}{b}{c}{d}$ and $\mul(A)=\mul\matriz{a}{b}{c}{d},$ by \cite[Corollary 2.2]{Hassi}, $A= \matriz{a}{b}{c}{d}.$ 

Conversely, suppose that $A$ is represented as $A=\matriz{a}{b}{c}{d}.$ 
Set  $a_0:=P_{\N_1}a,$ $b_0:=P_{\N_1}b,$ $c_0:=P_{\N_2}c$ and $d_0:=P_{\N_2}d.$ 
Then $a_0$ is an operator in $\N_1.$ In fact, if $(0,y) \in a_0,$ then there exists $z \in \St$ such that $(0,z) \in a$ and $y=P_{\N_1}z.$ 
Therefore, $z\in \mul(a) \subseteq \M_1 \perp \N_1$  and then $y=0.$ Analogously, $b_0,$ $c_0$  and $d_0$ are operators.
Also,
\begin{align*}
\dom \matriz{a_0}{b_0}{c_0}{d_0}&=\dom(a_0) \cap \dom(c_0) \oplus \dom(b_0) \cap \dom(d_0)\\
&= \dom(a) \cap \dom(c) \oplus \dom(b) \cap \dom(d)\\
&=\mc{D}_1 \oplus 
\mc{D}_2=\dom(A_0).
\end{align*}

Let $(x,y) \in \matriz{a_0}{b_0}{c_0}{d_0}.$ Then $x=x_1+x_2 \in \mc{D}_1 \oplus \mc{D}_2 \subseteq \cldom(A)$ and $y=\vecd{a_0x_1+b_0x_2}{c_0x_1+d_0x_2}\in \N_1 \oplus \N_2=\cldom(A).$ 

Set $w_1:=a_0x_1$ and $z_1:=b_0x_2.$  Then $(x_1,w_1) \in a_0=P_{\N_1}a$ and $(x_2,z_1) \in b_0=P_{\N_1}b.$ Then, there exists $s_1 \in \St$  such that $(x_1, s_1)\in a$ and $w_1=P_{\N_1}s_1,$ and there exists $t_1 \in \St$ such that $(x_2, t_1)\in b$ and $z_1=P_{\N_1}t_1.$ Recall that $\St=\N_1 \oplus \M_1$ then $P_{\N_1}+P_{\M_1}=P_{\St}$ so that $$w_1=P_{\N_1}s_1=s_1-P_{\M_1}s_1 \mbox { and } z_1=P_{\N_1}t_1=t_1-P_{\M_1}t_1.$$ Hence, since $P_{\M_1}s_1+P_{\M_1}t_1 \in \M_1 =\mul(a)+\mul(b),$ there exist $m_1 \in \mul(a)$ and $n_1 \in\mul(b)$ such that $P_{\M_1}s_1+P_{\M_1}t_1 =m_1+n_1.$ Then $(0,m_1) \in a$ and $(0,n_1) \in b.$ Therefore $w_1+z_1=(s_1-m_1)+(t_1-n_1)$ and 
$$(x_1,s_1-m_1) = (x_1,s_1) - (0, m_1) \in a \mbox{ and } (x_2,t_1-n_1) = (x_2,t_1) - (0, n_1) \in b.$$ 
Similarly, set $w_2:=c_0x_1$ and $z_2:=d_0x_2.$ Then, there exist $s_2, t_2 \in \St^{\perp},$ $m_2 \in \mul(c)$ and $n_2 \in \mul(d)$ such that $w_2+z_2=(s_2-m_2)+(t_2-n_2)$, $(x_1,s_2-m_2) \in c$  and $(x_2,t_2-n_2) \in d .$ 
Therefore,
$$(x,y)=\left(\vecd{x_1}{x_2}, \vecd{w_1+z_1}{w_2+z_2}\right)=\left(\vecd{x_1}{x_2}, \vecd{(s_1-m_1)+(t_1-n_1)}{(s_2-m_2)+(t_2-n_2)}\right) \in A.$$
Hence, $(x,y) \in A \cap (\cldom(A) \times \cldom(A))=A_0.$ Then,  $\matriz{a_0}{b_0}{c_0}{d_0} \subset A_0$ and, since  $\dom \matriz{a_0}{b_0}{c_0}{d_0} =\dom(A_0),$ it follows that $A_0=\matriz{a_0}{b_0}{c_0}{d_0}.$
\end{proof}



\begin{cor} \label{corMDLR3} Let  $A$ be a selfadjoint linear relation in $\HH,$ let $\St$ be a closed subspace of $\HH$ such that $P_\St(\dom(A)) \subseteq \dom(A)$ and suppose that $A$ admits the representation with respect to $\St\oplus \St^{\perp},$
$A=\matriz{a}{b}{c}{d}.$ If $\dom(a) \subseteq \dom(c)$ and $\mul(b) \subseteq \mul (a)$ then
	$$
	a=P_{\N_1}a \ \hat{\oplus } \ (\{0\} \times \mul(a)).
	$$
	Similar results can be stated for $b, c$ and $d.$
\end{cor}
\begin{proof} By Corollary \ref{corMDLR}, $A$ admits the representation with respect to $\St\oplus \St^{\perp},$
$$A=\matriz{P_{\N_1}a \ \hat{\oplus } \ (\{0\} \times \mul(a))}{P_{\N_1}b \ \hat{\oplus } \ (\{0\} \times \mul(b))}{P_{\N_2}c \ \hat{\oplus } \ (\{0\} \times \mul(c))}{P_{\N_2}d \ \hat{\oplus } \ (\{0\} \times \mul(d))}.$$ Set  $\tilde{a}:=P_{\N_1}a \ \hat{\oplus } \ (\{0\} \times \mul(a)),$ $\tilde{b}:=P_{\N_1}b \ \hat{\oplus } \ (\{0\} \times \mul(b)),$ $\tilde{c}:=P_{\N_2}c \ \hat{\oplus } \ (\{0\} \times \mul(c))$ and $\tilde{d}:=P_{\N_2}d \ \hat{\oplus } \ (\{0\} \times \mul(d)).$

Clearly, $\dom(a)=\dom(\tilde{a})$ and $\mul(a)=\mul(\tilde{a}).$ Let $(x,y) \in a$ then there exists $y' \in \St^{\perp}$ such that $(x,y') \in c$ because $\dom(a)\subseteq \dom(c).$ So that
$$(x,y)=\left(\vecd{x}{0}, \vecd{y+0}{y'+0}\right) \in A=\matriz{\tilde{a}}{\tilde{b}}{\tilde{c}}{\tilde{d}}.$$ Then
$(x,y)=\left(\vecd{x}{0}, \vecd{y}{y'}\right) =\left(\vecd{x}{0}, \vecd{w+z}{w'+z'}\right)$ with $(x,w) \in \tilde{a},$ $(x,w') \in 
\tilde{c},$  $(0,z)\in \tilde{b}$  and $(0,z') \in \tilde{d}.$ 

Then $(0,z) \in \mul(\tilde{b})=\mul(b) \subseteq \mul(a)=\mul(\tilde{a})$ so that, $(0,z) \in \tilde{a}.$
Hence $$(x,y)=(x,w+z)=(x,w)+(0,z) \in  \tilde{a}.$$ Then $a \subseteq \tilde{a}$ and since $\dom(a)=\dom(\tilde{a})$ and $\mul(a)=\mul(\tilde{a}),$ by \cite[Corollary 2.2]{Hassi}, $a=\tilde{a}=P_{\N_1}a \ \hat{\oplus } \ (\{0\} \times \mul(a)).$ The analogous results for $b, c$ and $d$  follow in a similar way.
\end{proof}


Next we focus on describing the matrix decompositions of nonnegative selfadjoint linear relations (operators).

The following lemmas are needed for the proof of Proposition \ref{proprepreA}.

\begin{lema} \label{lemapreA} Let $A$ be a selfadjoint linear relation in $\HH$ and let $\St$ be a closed subspace of $\HH$ such that $P_{\St} (\dom(A)) \subseteq \dom(A).$ Consider the matrix representation of $A$ as in \eqref{Acanon}. 
Then $a$ and $d$ are symmetric
linear relations, $c \subset b^{*}$ and $a, b, c$ and $d$ are decomposable linear relations with (unique) decompositions: $a=P_{\N_1}a \ \hat{\oplus } \ (\{0\} \times \M_1),$ $b=P_{\N_1}b \ \hat{\oplus } \ (\{0\} \times \M_1),$ $c=P_{\N_2}c \ \hat{\oplus } \ (\{0\} \times \M_2)$ and $d=P_{\N_2}d \ \hat{\oplus } \ (\{0\} \times \M_2).$
\end{lema}

\begin{proof} Let
\begin{equation*}\label{matrixA}
A=\matriz{a}{b}{c}{d}
\end{equation*} 
be the matrix representation of $A$ with respect to $\St \oplus \St^{\perp}$ given by Theorem \ref{theoMLR}. From Lemma \ref{lemaRMN2}, $\St=\N_1 \oplus \M_1$ and $\St^{\perp}=\N_2 \oplus \M_2.$ 	Write $A=A_0 \ \hat{\oplus} \ A_{\mul}$ as in \eqref{A0}. Then, by Corollary \ref{corMDLR}, $A_0$ admits the matrix representation with respect to $\N_1 \oplus \N_2$ 
\begin{equation}\label{matrixA01}
A_0=\matriz{a_0}{b_0}{c_0}{d_0},
\end{equation} 
where  $a_0:=P_{\N_1}a,$ $b_0:=P_{\N_1}b,$ $c_0:=P_{\N_2}c$ and $d_0:=P_{\N_2}d.$ 
Since $\dom(a)=\dom(c)=\mc{D}_1$ and $\mul(a)=\mul(b)=\M_1,$ by Corollary \ref{corMDLR3}, $a=a_0 \ \hat{\oplus } \ (\{0\} \times \M_1).$ Likewise, $b=b_0 \ \hat{\oplus } \ (\{0\} \times \M_1),$ $c=c_0 \ \hat{\oplus } \ (\{0\} \times \M_2)$ and $d=d_0 \ \hat{\oplus } \ (\{0\} \times \M_2).$

Define $$\hat{A}_0:=\matriz{a_0^\times}{c_0^\times}{b_0^\times}{d_0^\times}$$ with $\dom(\hat{A}_0)=  \dom(a_0^\times) \cap \dom(b_0^\times) \oplus \dom(c_0^\times)\cap \dom(d_0^\times),$ where $a_0^{\times}$ denotes the adjoint of $a_0$ when viewed as an operator from $\N_1$ to $\N_1,$ likewise $b_0^{\times}, c_0^{\times}$ and $d_0^{\times}.$

Since $A$ is selfadjoint, $A_0=A_0^{\times},$ where $A_0^{\times}$ denotes the adjoint of $A_0$ when viewed as an operator from $\cldom(A)$ to $\cldom(A).$ Then  $A_0^\times$ admits a matrix decomposition with respect to $\N_1 \oplus \N_2$. Then, by \cite[Theorem 2.2]{Chen},  $A_0=A_0^\times=\hat{A}_0.$
So that
\begin{equation*} \label{eqaux1}
	A_0=\matriz{a_0}{b_0}{c_0}{d_0}=\matriz{a_0^\times}{c_0^\times}{b_0^\times}{d_0^\times}=\hat{A}_0.
\end{equation*}

Then   $$a_0 \subset  a_0^\times, \ \ d_0 \subset  d_0^\times, \ \ b_0 \subset c_0^\times  \mbox{ and } \ \ c_0 \subset b_0^\times.$$ 

So that $a_0$ and $d_0$ are  symmetric operators on $\N_1$ and $\N_2,$ respectively, and $b_0$ and $c_0$ are closable operators. Also, since $a_0, b_0, c_0, d_0$ are closable operators, by Theorem \ref{TdecompHassi}, $a,b, c$ and $d$ are decomposable with (unique) decompositions: $a=P_{\N_1}a \ \hat{\oplus } \ (\{0\} \times \M_1),$ $b=P_{\N_1}b \ \hat{\oplus } \ (\{0\} \times \M_1),$ $c=P_{\N_2}c \ \hat{\oplus } \ (\{0\} \times \M_2)$ and $d=P_{\N_2}d \ \hat{\oplus } \ (\{0\} \times \M_2).$

Let us see that $a \subset a^*.$ Let $(x_1,w_1)\in a,$ then $x_1 \in \mc{D}_1$ and there exists $m_1 \in \M_1$ such that
$$(x_1,w_1)=(x_1,a_0x_1)+(0,m_1).$$
Also, let $(f,g)\in a,$ then $f \in \mc{D}_1$ and there exists $m \in \M_1$ such that
$$(f,g)=(f,a_0f)+(0,m).$$ Hence
\begin{align*}
\PI{g}{x_1}_\HH&=\PI{a_0f+m}{x_1}_\HH=\PI{a_0f}{x_1}_\HH=\PI{a_0f}{x_1}_{\N_1}\\
&=\PI{a_0^{\times}f}{x_1}_{\N_1}=\PI{f}{a_0x_1}_{\N_1}=\PI{f}{a_0x_1+m_1}_\HH=\PI{f}{w_1}_\HH.
\end{align*}
Then $(x_1,w_1) \in a^*.$ Likewise, $d \subset d^*$ and $c \subset b^*.$ 
\end{proof}

By the proof of the last lemma, $A \subset \matriz{a^*}{c^*}{b^*}{d^*}$ and, by \cite[Proposition 6.1]{HassiMR}, the other inclusion always holds. So that $A$ admits the matrix representation
$$A=\matriz{a^*}{c^*}{b^*}{d^*}.$$

\begin{lema} [\normalfont{cf. \cite[Chapter VI]{Kato}, \cite[Lemma 5.3.1]{Behrndt}}]  \label{lemaFr} Let $A$ be a nonnegative symmetric linear relation in $\HH.$ If $A_F$ is the Friedrichs extension of $A$, then  $\dom(A)$ is a core of $A_F^{1/2}$ and $\mul(A_F)=\mul(A^*).$ 
\end{lema}

\begin{prop}\label{proprepreA} Let $A \geq 0$ be a linear relation in $\HH$ and let $\St$ be a closed subspace of $\HH$ such that $P_{\St} (\dom(A)) \subseteq \dom(A).$ Then $A$ admits the $2 \times 2$ block matrix representation with respect to $\St \oplus \St^{\perp}$
\begin{equation}\label{repreA}
\matriz{a_F}{b}{c}{d_F}
\end{equation}
where $a_F$ and $d_F$ are the Friedrichs extensions of $a:=P_{\St}A|_{\St}$ and $d:=P_{\St^{\perp}}A|_{\St^{\perp}},$ respectively, $b:=P_{\St}A|_{\St^{\perp}},$  $c:=P_{\St^{\perp}}A|_{\St}$ are decomposable linear relations and $c \subset b^{*}.$ 

Moreover, if $A$ is written as in \eqref{A0} then 
$A_0$ admits the matrix representation with respect to $\N_1 \oplus \N_2:$ 
\begin{equation}\label{repreA01}
A_0=\matriz{(a_F)_0}{b_0}{c_0}{(d_F)_0}
\end{equation}
where $(a_F)_0$ and $(d_F)_0$ are the nonnegative selfadjoint operator parts of $a_F$ and $d_F,$ respectively and
$a_F=(a_F)_0\ \hat{\oplus } \ (\{0\} \times \M_1),$ $b=b_0 \ \hat{\oplus } \ (\{0\} \times \M_1),$ $c=c_0\ \hat{\oplus } \ (\{0\} \times \M_2)$ and $d=(d_F)_0 \ \hat{\oplus } \ (\{0\} \times \M_2),$  where $b_0=P_{\N_1}b,$  $c_0=P_{\N_2}c$ and $(a_F)_0$ and $(d_F)_0$ are  the Friedrichs extensions of $a_0=P_{\N_1}a$ and $d_0=P_{\N_2}d,$ respectively.
\end{prop}

\begin{proof} Let
\begin{equation*}
A=\matriz{a}{b}{c}{d}
\end{equation*} 
be the matrix representation of $A$ with respect to $\St \oplus \St^{\perp}$ as in Lemma \ref{lemapreA}. Since $A \geq 0,$ it follows that $a$ and $d$ are nonnegative symmetric linear relations.

Also, by Corollaries \ref{corMDLR} and \ref{corMDLR3}, if $A$ is written as in \eqref{A0} then 
$A_0$ admits the matrix representation with respect to $\N_1 \oplus \N_2:$ 
$
A_0=\matriz{a_0}{b_0}{c_0}{d_0},$
where
$a=a_0\ \hat{\oplus } \ (\{0\} \times \M_1),$ $b=b_0 \ \hat{\oplus } \ (\{0\} \times \M_1),$ $c=c_0 \ \hat{\oplus } \ (\{0\} \times \M_2)$ and $d=d_0\ \hat{\oplus } \ (\{0\} \times \M_2).$

 Let $a_F$ and $d_F$ be the Friedrichs extensions of $a$ and $d,$ respectively.  By Lemma \ref{lemaFr}, $\dom(a)=\mc{D}_1$ is a core of $a_F^{1/2}$ and $\dom(d)=\mc{D}_2$ is a core of $d_F^{1/2}.$ 

Set $$A':=\matriz{a_F}{b}{c}{d_F}.$$ Then $\dom(A')=\dom(a_F)\cap \dom(c) \oplus \dom(b)\cap \dom(d_F)=\mc{D}_1\oplus \mc{D}_2=\dom(A),$ because $\dom(c)=\mc{D}_1$ and $\dom(b)=\mc{D}_2.$ Also,
$\mul(A')=\mul(a_F)+ \mul(b) \oplus \mul(c) + \mul(d_F)=\M_1 \oplus \M_2=\mul(A),$ because $\mul(a_F)=\mul(a^*)=\dom(a)^{\perp}= \M_1,$ $\mul(b)=\M_1,$ $\mul(d_F)=\mul(d^*)= \dom(d)^{\perp}=\M_2$ and $\mul(c)=\M_2.$ 
But, since
$A \subset A',$ it follows that $$A=A'=\matriz{a_F}{b}{c}{d_F}.$$
Since $a_F$ and $d_F$ are selfadjoint, $a_F$ and $d_F$ are decomposable and $a_F=(a_F)_0 \ \hat{\oplus} \ (\{0\}  \times \M_1)$ and $d_F=(d_F)_0 \ \hat{\oplus} \ (\{0\}  \times \M_2)$ where $(a_F)_0$ and $(d_F)_0$ are the nonnegative selfadjoint operator parts of $a_F$ and $d_F,$ respectively.

Let us see that $(a_F)_0$ is the Friedrichs extension of $a_0$ and $(d_F)_0$ is the Friedrichs extension of $d_0,$ cf. \cite[Theorem 5.3.3]{Behrndt}. Since $a$ is a nonnegative symmetric linear relation in $\St,$  the form $\mathfrak{t}_a$ given by $\mathfrak{t}_a[u,v]:=\PI{u'}{v}$ for $(u,u'), (v,v') \in a$ with $\dom(\mathfrak{t}_a)=\dom(a),$ is nonnegative and closable, \cite[Lemma 5.1.17]{Behrndt}. Also, by the proof of Lemma \ref{lemapreA}, $a_0$ is a nonnegative symmetric linear operator on $\N_1,$ then the form $\mathfrak{t}_{a_0}$ given by $\mathfrak{t}_{a_0}[u,v]:=\PI{a_0u}{v}$ for $u,v \in \dom(a_0),$ with $\dom(\mathfrak{t}_{a_0})=\dom(a_0),$ is nonnegative and closable. But
$$\mathfrak{t}_a=\mathfrak{t}_{a_0}.$$ In fact, it is clear that $\dom(\mathfrak{t}_{a_0})=\dom(\mathfrak{t}_{a}).$ Let $u,v \in  \dom(\mathfrak{t}_{a})=\dom(a)$ then there exist $u', v' \in \HH$ such that $(u,u'), (v,v') \in a.$ Then $u'=a_0u+m$ for some $m \in \M_1 \perp \N_1.$ Then
$$\mathfrak{t}_a[u,v]=\PI{a_0u+m}{v}=\PI{a_0u}{v}=\mathfrak{t}_{a_0}[u,v],$$ because $v \in \mc{D}_1.$
Hence, the closures of the forms coincide, i.e., $\ol{\mathfrak{t}_a}=\ol{\mathfrak{t}_{a_0}}.$ 
Then, by the Second Representation Theorem \cite[Theorem 5.1.23]{Behrndt}, 
$$\ol{\mathfrak{t}_a}[u,v]=\PI{(a_F)_0^{1/2}u}{(a_F)_0^{1/2}v}$$ for every $u,v \in \dom((a_F)_0^{1/2})=\dom(\ol{\mathfrak{t}_a})$ and 
$$\ol{\mathfrak{t}_{a_0}}[u,v]=\PI{(a_0)_F^{1/2}u}{(a_0)_F^{1/2}v}$$ for every $u,v \in \dom((a_0)_F^{1/2})=\dom(\ol{\mathfrak{t}_{a_0}}),$ where $(a_0)_F$ is the Friedrichs extension of $a_0.$ So that $(a_F)_0=(a_0)_F.$
Likewise, $(d_F)_0=(d_0)_F.$
Then, $a_0 \subset(a_F)_0,$ $d_0 \subset(d_F)_0$ and, by Lemma \ref{lemaFr}, $\dom(a_0)=\mc{D}_1$ is a core of $(a_F)_0^{1/2}$ and $\dom(d_0)=\mc{D}_2$ is a core of $(d_F)_0^{1/2}.$ 
Then
$$A_0 \subset A'':=\matriz{(a_F)_0}{b_0}{c_0}{(d_F)_0}.
$$
But, $\dom(A'')=\dom((a_F)_0)\cap \dom(c_0) \oplus \dom(b_0)\cap \dom((d_F)_0)=\mc{D}_1\oplus \mc{D}_2=\dom(A_0),$ because $\dom(c_0)=\mc{D}_1$ and $\dom(b_0)=\mc{D}_2.$ Then $A_0=A''.$
\end{proof}

\begin{thm} \label{thmRMRL} Let  $A \geq 0$ be a  linear relation in $\HH$ and let $\St$ be a closed subspace of $\HH$ such that $P_{\St} (\dom(A)) \subseteq \dom(A).$ Then $A$ admits a matrix decomposition in $\HH$ with respect to $\St \oplus \St^{\perp},$
	\begin{equation}\label{repreARL}
		A=\matriz{a}{b}{c}{d},
	\end{equation}
such that:
	\begin{enumerate}
		\item[1.]   $a$ and $d$  are nonnegative selfadjoint linear relations with $\mc{D}_1\subseteq \dom(a),$ $\mc{D}_2\subseteq \dom(d),$ $\mc{D}_2=\dom(b),$ $\mc{D}_1=\dom(c),$ and $c \subset b^{*};$ 
		\item[2.]  $\mc{D}_1$ is a core of $a^{1/2}$ and  $\mc{D}_2$ is a core of $d^{1/2};$
		\item [3.] there exists a contraction $g: \St^{\perp}\ra \St$ such that $$b=a^{1/2}gd^{1/2}|_{\mc{D}_2} \mbox{ and } c=d^{1/2}g^*a^{1/2}|_{\mc{D}_1}.$$
	\end{enumerate}
\end{thm}

\begin{proof} Items $1$ and $2$ are proved in Proposition \ref{proprepreA}. 

$3:$ Let $A=\matriz{a}{b}{c}{d}$ be  the block matrix representation of $A$ given in \eqref{repreA}. From Lemma \ref{lemaRMN2}, $\St=\N_1 \oplus \M_1$ and $\St^{\perp}=\N_2 \oplus \M_2.$ Write $A=A_0 \ \hat{\oplus} \ A_{\mul}$ as in \eqref{A0}. Then, by Proposition \ref{proprepreA}, $A_0$ admits the matrix representation with respect to $\N_1 \oplus \N_2:$ 
$$
A_0=\matriz{a_0}{b_0}{c_0}{a_0},
$$
where $a_0$ and $d_0$ are the nonnegative selfadjoint operator parts of $a$ and $d,$ respectively, $\mc{D}_1$ is a core of $a_0^{1/2},$ $\mc{D}_2$ is a core of $d_0^{1/2},$ 
$a=a_0\ \hat{\oplus } \ (\{0\} \times \M_1),$ $b=b_0 \ \hat{\oplus } \ (\{0\} \times \M_1),$ $c=c_0\ \hat{\oplus } \ (\{0\} \times \M_2)$ and $d=d_0 \ \hat{\oplus } \ (\{0\} \times \M_2).$

Since $A \geq 0,$ then $A_0$ is a nonnegative selfadjoint operator on $\cldom(A).$ 
Then
$$\PI{A_0^{1/2}h}{A_0^{1/2}k}=\PI{A_0h}{k} \mbox{ for every } h, k \in \dom(A),$$ because $A_0=A_0^{1/2}A_0^{1/2}.$ In particular, for every  $h_1 \in \mc{D}_1$
	\begin{align*}
		\PI{A_0^{1/2}h_1}{A_0^{1/2}h_1}&=\PI{A_0 h_1}{h_1}=\PI{a_0h_1}{h_1}=\PI{a_0^{1/2}h_1}{a_0^{1/2}h_1}.
	\end{align*}
	Then the map $a_0^{1/2}(\mc{D}_1) \ra  A_0^{1/2}(\mc{D}_1),$  $$a_0^{1/2}h_1 \mapsto A_0^{1/2}h_1$$  can be extended to a partial isometry $V_1$ on all of $\N_1,$ with initial space $\ol{a_0^{1/2}(\mc{D}_1)}=\clran(a_0^{1/2})$ (where we used that $\mc{D}_1$ is a core of $a_0^{1/2}$), so that $\ker(V_1)=\ker(a_0^{1/2}),$ and final space $\ol{A_0^{1/2}(\mc{D}_1)}.$  Therefore
	\begin{equation}
		V_1a_0^{1/2} =A_0^{1/2}\mbox{ on } \mc{D}_1.
	\end{equation}
	So, for  every $h_2 \in \mc{D}_2$ and $k_1 \in \mc{D}_1,$ 
	\begin{align*}
		\PI{b_0h_2}{k_1} &=\PI{A_0h_2}{k_1}=\PI{A_0^{1/2}h_2}{A_0^{1/2}k_1}=\PI{A_0^{1/2}h_2}{V_1a_0^{1/2}k_1}\\
		&=\PI{{V_1}^*A_0^{1/2}h_2}{a_0^{1/2}k_1}.
	\end{align*}
	Therefore, ${V_1}^*A_0^{1/2}h_2 \in \dom((a_0^{1/2})^\times)$ and $(a_0^{1/2})^\times{V_1}^*A_0^{1/2}h_2=b_0h_2.$ Since $a_0^{1/2}$ is selfadjoint and the above holds for any $h_2 \in \mc{D}_2,$ it follows that
	$$b_0=a_0^{1/2}V_1^*A_0^{1/2} \mbox { on } \mc{D}_2.$$
	Likewise, there exists a partial isometry $V_2$ in $\N_2$ with initial space $\ol{d_0^{1/2}(\mc{D}_2)}$ and final space $\ol{A_0^{1/2}(\mc{D}_2)},$ such that
	$$V_2d_0^{1/2}=A_0^{1/2} \mbox{ on } \mc{D}_2  \mbox{ and } c_0=d_0^{1/2} V_2^* A_0^{1/2} \mbox{ on }  \mc{D}_1.$$ 
	Then $$b_0h_2=a_0^{1/2} {V_1}^* A_0^{1/2}h_2=a_0^{1/2} {V_1}^*V_2 d_0^{1/2}h_2 \mbox{ for every } h_2 \in \mc{D}_2.$$ 	Set $f:={V_1}^* V_2.$ Then $f$ is a contraction from $\N_1$ to $\N_2$ such that $b_0=a_0^{1/2} f d_0^{1/2}$ on $\mc{D}_2.$ Likewise,
	$c_0=d_0^{1/2}f^*a_0^{1/2}$ on $\mc{D}_1.$

	 Using that $\St^{\perp}=\N_2 \oplus \M_2,$ $f$ has an extension, again a contraction from $\St^{\perp}$ to $\St$, named $g$ such that 
	 $gx=0$ for every $x \in \M_2.$ 
	Let $(x,y) \in a^{1/2} gd^{1/2}|_{\mc{D}_2}.$ Then there exists $z \in \St^{\perp}$ such that $(x,z) \in d^{1/2}|_{\mc{D}_2}$ and $(z,y) \in a^{1/2} g.$ Then $$(x,z)=(x,d_0^{1/2}x)+(0,m_2)$$ for some $m_2 \in \M_2$ and so $z=d_0^{1/2}x+m_2.$ Also, since $(z,y) \in a^{1/2} g,$ it follows that $(gz,y)  \in a^{1/2}.$ Then 
	$(gz,y)=(gz,a_0^{1/2}gz)+(0,m_1)$ for some $m_1 \in \M_1.$ Then, since $m_2 \in \ker(g)$ and $d_0^{1/2}x \in \N_2,$
	$$y=a_0^{1/2}gz+m_1=a_0^{1/2}g(d_0^{1/2}x+m_2)+m_1=a_0^{1/2}fd_0^{1/2}x+m_1=b_0x+m_1.$$ Hence,
	$$(x,y)=(x,b_0x)+(0,m_1) \in b.$$ Conversely, suppose that $(x,y) \in b,$ then  $x \in \mc{D}_2$ and $$(x,y)=(x,b_0x)+(0,m_1)=(x,a_0^{1/2}fd_0^{1/2}x)+(0,m_1)$$ for some $m_1 \in \M_1$ and so $y=a_0^{1/2}fd_0^{1/2}x+m_1.$ Set $z:=d_0^{1/2}x \in \N_2$ then $(x,z)= (x, d_0^{1/2}x) \in d^{1/2}|_{\mc{D}_2}.$ Also, 
	\begin{align*}
		(gz,y)&=(gz,a_0^{1/2}fd_0^{1/2}x)+ (0,m_1)=(gz,a_0^{1/2}fz)+ (0,m_1)\\
		&=(gz,a_0^{1/2}gz)+ (0,m_1) \in a^{1/2}.
	\end{align*} So that $(z,y) \in a^{1/2}  g$ and then $(x,y) \in  a^{1/2}gd^{1/2}|_{\mc{D}_2}.$ 

Likewise,  $c=d^{1/2}g^*a^{1/2}|_{\mc{D}_1}.$
\end{proof}

\bigskip
\begin{cor} \label{propDP} Let $A \geq 0$ be a linear operator in $\HH$ and let $\St$ be a closed subspace of $\HH$ such that $P_{\St} (\dom(A)) \subseteq \dom(A).$ Let $A=\matriz{a}{b}{c}{d}$ be the block matrix representation of $A$ given in \eqref{repreA}. Set $Z:=\matriz{a^{1/2}|_{\mc{D}_1}}{0}{0}{d^{1/2}|_{\mc{D}_2}}$ and  $W:=\matriz{1}{f}{0}{(1-f^*f)^{1/2}} \in L(\HH),$ where $f: \St^{\perp}\ra \St$  is the contraction in the proof of Theorem \ref{thmRMRL}. Then the operator $WZ$ is closable and
$$A=(WZ)^*WZ=(WZ)^*\ol{WZ}.$$
\end{cor}

\begin{proof} Define  $\Gamma:=W^*W=\matriz{1}{f}{f^*}{1}.$ Then $\Gamma \in L(\HH)$ and $\Gamma \geq 0,$  because $f$ is a contraction, and $Z$ is a densely defined operator with $\dom(Z)=\mc{D}_1 \oplus \mc{D}_2.$ Since $\mc{D}_1$ is a core of $a^{1/2}$ and $\mc{D}_2$ is a core of $d^{1/2},$  
$$Z^*=\matriz{a^{1/2}}{0}{0}{d^{1/2}}.$$
Consider the operator $Z^*\Gamma Z.$ Then $$\dom(Z^*\Gamma Z)=\mc{D}_1 \oplus \mc{D}_2.$$
Clearly, $\dom(Z^*\Gamma Z)\subseteq \dom(Z)=\mc{D}_1 \oplus \mc{D}_2.$
On the other hand, take $h=\vecd{h_1}{h_2} \in \mc{D}_1 \oplus \mc{D}_2,$ then
$$\Gamma Z\vecd{h_1}{h_2}=\matriz{1}{f}{f^*}{1}  \vecd{a^{1/2}h_1}{d^{1/2}h_2}=\vecd{a^{1/2}h_1+fd^{1/2}h_2}{f^*a^{1/2}h_1+d^{1/2}h_2}.$$
Since $b=a^{1/2}fd^{1/2}$ on $\mc{D}_2$ and $a^{1/2}(\mc{D}_1) \subseteq \dom(a^{1/2}),$ it follows that  $a^{1/2}h_1+fd^{1/2}h_2 \in \dom(a^{1/2}).$ Likewise, since
$c=d^{1/2}f^*a^{1/2}$ on $\mc{D}_1$ and $d^{1/2}(\mc{D}_2) \subseteq  \dom(d^{1/2}),$ it follows that  $f^*a^{1/2}h_1+d^{1/2}h_2 \in \dom(d^{1/2}).$ Hence, $\Gamma Zh \in \dom(Z^*)$ and $h \in \dom(Z^*\Gamma Z).$ Then $Z^*\Gamma Z$ has matrix representation and, by \cite[Theorem 2.1]{Chen}, 
\begin{align*}
	Z^*\Gamma Z&=\matriz{a^{1/2}}{0}{0}{d^{1/2}} \matriz{1}{f}{f^*}{1}\matriz{a^{1/2}|_{\mc{D}_1}}{0}{0}{d^{1/2}|_{\mc{D}_2}}\\
	&=\matriz{a|_{\mc{D}_1}}{b}{c}{d|_{\mc{D}_2}} \subseteq A.
\end{align*}
But, since $\dom(Z^*\Gamma Z)=\dom(A)$ it follows that $A=Z^*\Gamma Z=Z^*W^*WZ=(WZ)^*WZ.$

If $Y:=WZ,$ then $\dom(Y)=\dom(Z)=\dom(A).$ 
Therefore, $\dom(Y^*Y)=\dom(A)=\dom(Y).$ Then, by \cite[Theorem 5.1]{Seb}, $Y=WZ$ is closable. Finally,
$$A=Y^*Y=A^*=(Y^*Y)^*\supset Y^*\ol{Y}\supset Y^*Y=A.$$
\end{proof}

\section{The Schur complement of nonnegative selfadjoint  linear relations}
Let $A \geq 0$ be a linear relation in $\HH$ and let $\St$ be a closed subspace of $\HH$ such that $P_{\St} (\dom(A)) \subseteq \dom(A).$  Let  
\begin{equation}\label{repreAS}
A=\matriz{a}{b}{c}{d}
\end{equation}  
be the $2 \times 2$ block matrix representation of $A$ with respect to $\St \oplus \St^{\perp}$ as in Theorem \ref{thmRMRL}. That is, $a$ and $d$  are nonnegative selfadjoint   linear relations with $\mc{D}_1\subseteq \dom(a),$ $\mc{D}_2\subseteq \dom(d),$ $\mc{D}_2=\dom(b),$ $\mc{D}_1=\dom(c),$ and $c \subset b^{*}.$ 
Also, $\mc{D}_1$ is a core of $a^{1/2},$ $\mc{D}_2$ is a core of $d^{1/2}$ and there exists a contraction $g: \St^{\perp}\ra \St$ such that $$b=a^{1/2}gd^{1/2}|_{\mc{D}_2} \mbox{ and } c=d^{1/2}g^*a^{1/2}|_{\mc{D}_1}.$$ Write $A=A_0 \ \hat{\oplus} \ A_{\mul}$ as in \eqref{A0}. Then, 
$a=a_0 \ \hat{\oplus } \ (\{0\} \times \M_1),$ $b=b_0 \ \hat{\oplus } \ (\{0\} \times \M_1),$ $c=c_0 \ \hat{\oplus } \ (\{0\} \times \M_2)$ and $d=d_0 \ \hat{\oplus } \ (\{0\} \times \M_2),$ where   
\begin{equation}\label{repreA0}
A_0=\matriz{a_0}{b_0}{c_0}{d_0}
\end{equation} is the $2 \times 2$ block matrix representation of $A_0$ with respect to $\N_1 \oplus \N_2$ given in \eqref{repreA01}. 
By Theorem \ref{thmRMRL}, there exists a contraction $f: \N_2 \ra \N_1$ such that $$b_0=a_0^{1/2}fd_0^{1/2}|_{\mc{D}_2} \mbox{ and } c_0=d_0^{1/2}f^*a_0^{1/2}|_{\mc{D}_1}.$$

By Lemma \ref{lemaRMN2}, $\St=\N_1 \oplus \M_1$ and $\St^{\perp}=\N_2 \oplus \M_2.$  Then $g=\matriz{f}{0}{0}{0}$ is the matrix decomposition of $g: \N_2 \oplus \M_2 \ra \N_1 \oplus \M_1.$  

In order to define the Schur complement of $A,$ consider $D_g:=(1-g^*g)^{1/2} \in L(\St^{\perp})$ and the closed linear relation
$$T:=\ol{D_gd^{1/2}|_{\mc{D}_2}} \subseteq \St^{\perp}\times \St^{\perp}.$$ 

\begin{lema}\label{lemaRL1} Under the above hypotheses, 
$$T^*T=d_0^{1/2}D_f \ol{D_fd_0^{1/2}|_{\mc{D}_2}}\ \hat{\oplus} \ ( \{0\} \times\M_2),$$
where $D_f:=(1-f^*f)^{1/2} \in L(\N_2).$
\end{lema}

\begin{proof} The matrix decomposition of $D_g$ with respect to $\N_2 \oplus \M_2$ is $D_g=\matriz{D_f}{0}{0}{1}.$  Then $D_gd_0^{1/2}|_{\mc{D}_2}=D_fd_0^{1/2}|_{\mc{D}_2} \subseteq \N_2 \times \N_2$ and, since $d^{1/2}|_{\mc{D}_2}=d_0^{1/2}|_{\mc{D}_2} \ \hat{\oplus} \ (\{0\} \times \M_2),$
\begin{equation}\label{eqDg}
D_gd^{1/2}|_{\mc{D}_2}=D_fd_0^{1/2}|_{\mc{D}_2} \ \hat{\oplus} \ (\{0\} \times \M_2).
\end{equation} 
So that 
\begin{equation}\label{eqT}
T=\ol{D_fd_0^{1/2}|_{\mc{D}_2}  \ \hat{\oplus} \ (\{0\} \times \M_2)}=\ol{t} \ \hat{\oplus} \ ( \{0\} \times\M_2),
\end{equation} 
where $t:=D_fd_0^{1/2}|_{\mc{D}_2}.$ Since $\mc{D}_2 \subseteq \dom(T)= \dom(\ol{t}) \subseteq \N_2,$ then 
$$\cldom(T)=\cldom(\ol{t})=\N_2.$$ 	
Also, $$T^*=(D_gd^{1/2}|_{\mc{D}_2})^*=(d^{1/2}|_{\mc{D}_2})^*D_g=d^{1/2}D_g,$$ where we used that $D_g \in L(\St^{\perp}$) so there is equality in \eqref{product}  and $\mc{D}_2$ is a core of $d^{1/2}.$
Then
$$T^*=(d_0^{1/2} \ \hat{\oplus} \ (\{0\} \times \M_2))D_g=d_0^{1/2}D_f\hat{\oplus} \ ( \{0\} \times\M_2)=t^{\times} \ \hat{\oplus} \ ( \{0\} \times\M_2),$$
where $t^{\times}$ denotes the adjoint of $t$ when viewed as an operator in $\N_2.$
Finally, since $t$ is a densely defined operator in $\N_2,$  $t^{\times}$ is an operator in $\N_2$ and $\mul(t^{\times}\ol{t})=\mul(t^{\times})=\{0\}.$ Therefore, by Theorem \ref{thmVN},  $t^{\times}\ol{t}$ is a nonnegative selfadjoint linear operator in $\N_2$ and
$$\mul(T^*T)=\mul(T^*)=\dom(T)^{\perp}=\St^{\perp} \ominus \N_2 =\M_2.$$ 
Now, suppose that $(x,y) \in T^*T.$ Then $(x,z) \in T$ and $(z,y) \in T^*$ for some $z \in \St^{\perp}.$ Then
\begin{align*}
	(x,z)&=(x,z')+(0,m) \mbox{ for some } m \in \M_2 \mbox{ and } z' \in \N_2 \mbox{ such that } (x,z') \in \ol{t},\\
	(z,y)&=(z, t^{\times}z)+(0, m') \mbox{ for some } m' \in \M_2.
\end{align*}
Since $z \in \dom(T^*) \subseteq \N_2,$ $z' \in \ran(\ol{t}) \subseteq \N_2$ and $z=z'+m,$ it holds  that $m=0$ and $z=z'.$
Then, from the fact that $(x,z)=(x,z') \in \ol{t}$ and  $(z, t^{\times}z) \in t^{\times}$ it follows that $(x,t^{\times}z) \in t^{\times} \ol{t}.$ Hence, since $y=t^{\times}z+m',$ 
$$(x,y)=(x,t^{\times}z)+(0,m') \in t^{\times}\ol{t}  \ \hat{\oplus} ( \{0\} \times\M_2).$$ Therefore
\begin{equation} \label{T*T}
	T^*T \subset t^{\times}\ol{t}  \ \hat{\oplus} \ (\{0\}  \times\M_2).
\end{equation}
By Theorem \ref{thmVN},  $T^*T$ is a nonnegative selfadjoint linear relation in $\St^{\perp}.$ Then $T^*T$  admits a unique decomposition as in \eqref{Tdecom}: 
$$T^*T=(T^*T)_0 \ \hat{\oplus} \ (\{0\} \times \M_2),$$
where $(T^*T)_0$ is a selfadjoint operator in $\cldom(T^*T)=\N_2.$
By \eqref{T*T}, $(T^*T)_0 \subset t^{\times}\ol{t}$ and, since $(T^*T)_0$ and  $t^{\times}\ol{t}$ are selfadjoint operators in $\N_2,$ equality holds, i.e., $(T^*T)_0=t^{\times}\ol{t}.$ Hence
$$T^*T=(T^*T)_0 \ \hat{\oplus} \ (\{0\} \times \M_2)=d_0^{1/2}D_f \ol{D_fd_0^{1/2}|_{\mc{D}_2}}\ \hat{\oplus} \ ( \{0\} \times\M_2).$$
\end{proof}

Consider the set
$$\mc{M}(A,\St^{\perp}):=\{X \mbox{ l.r. in } \HH: 0 \leq X \leq A,  \ran(X) \subseteq \St^{\perp}\}.$$
In \cite{Arlinskii}, Arlinski\u{\i } proved that the set $\mc{M}(A,\St^{\perp})$ has a maximum element and defined the Schur complement of $A$ to $\St$ denoted by $A_{/ \St}$ as the maximum of $\mc{M}(A,\St^{\perp}).$
In what follows we give an alternate proof of the existence of the Schur complement as well as a formula for $A_{/ \St}$  using the matrix decomposition of $A$ when $P_{\St} (\dom(A)) \subseteq \dom(A).$ 

\begin{thm} \label{thmSchur} Let $A$ be a linear relation in $\HH,$ let $\St$ be a closed subspace of $\HH$ such that $P_{\St} (\dom(A)) \subseteq \dom(A)$ and consider the matrix representation of $A$ with respect to $\St \oplus \St^{\perp}$  in \eqref{repreAS}. Then the set $\mc{M}(A,\St^{\perp})$ has a maximum element $A_{/ \St}.$ Moreover, 
$$A_{/ \St}=\matriz{0}{0}{0}{T^*T},$$ where $T:=\ol{D_gd^{1/2}|_{\mc{D}_2}}.$
\end{thm}	

\begin{proof} Write $A=A_0 \ \hat{\oplus} \ A_{\mul(A)}$ and set $C:=\matriz{0}{0}{0}{T^*T}.$ Then $\ran(C)=\ran(T^*T)=\ran((T^*T)_0) \oplus \M_2 \subseteq \N_2 \oplus \M_2= \St^{\perp}$ and $C^*=C \geq 0.$ 
Suppose that $T$ is written as $T=T_0 \ \hat{\oplus} \ (\{0\} \times \mul(T))$ as in \eqref{Tdecom}. 
Let $C_0$ be the operator part of $C$ then, by \cite[Proposition 2.7]{Hassi}, 
$$\PI{C_0^{1/2}u}{C_0^{1/2}v}=\PI{T_0u_2}{T_0v_2},$$ 
for every  $u=\vecd{u_1}{u_2}, v=\vecd{v_1}{v_2} \in \dom(C_0^{1/2})=\St \oplus \dom(T_0).$

Then, since $\mc{D}_2 \subseteq \dom(T)=\dom(T_0)$ 
$$
\dom(A)=\mc{D}_1 \oplus \mc{D}_2 \subseteq \St \oplus \dom(T_0)=\dom(C_0^{1/2}).	
$$
	
Let \eqref{repreA0} be the matrix decomposition of $A_0$ (in $\cldom(A)$) with respect to $\N_1 \oplus \N_2.$ 
Let $V_1$ and $V_2$ be the partial isometries given in the proof of Theorem \ref{thmRMRL} such that
$$V_1a_0^{1/2}=A_0^{1/2} \mbox{ on } \mc{D}_1 \mbox{ and } V_2d_0^{1/2}=A_0^{1/2} \mbox{ on } \mc{D}_2,$$
and  $f={V_1}^*V_2.$
Then, by Corollary \ref{propDP}, $A_0=Z^*\Gamma Z,$ where $\Gamma=\matriz{1}{f}{f^*}{1}$ and $Z=\matriz{a_0^{1/2}|_{\mc{D}_1}}{0}{0}{d_0^{1/2}|_{\mc{D}_2}}.$ 
Let $h=\vecd{h_1}{h_2} \in \mc{D}_1 \oplus \mc{D}_2.$ Then
\begin{align*}
\PI{A_0h}{h}&=\PI{\matriz{1}{f}{f^*}{1}\vecd{a_0^{1/2}h_1}{d_0^{1/2}h_2}}{\vecd{a_0^{1/2}h_1}{d_0^{1/2}h_2}}\\
& \geq  \PI{\matriz{1}{f}{f^*}{1}_{/ \N_1}\vecd{a_0^{1/2}h_1}{d_0^{1/2}h_2}}{\vecd{a_0^{1/2}h_1}{d_0^{1/2}h_2}}\\
&=\PI{\matriz{0}{0}{0}{1-f^*f}\vecd{a_0^{1/2}h_1}{d_0^{1/2}h_2}}{\vecd{a_0^{1/2}h_1}{d_0^{1/2}h_2}}\\
&=\PI{D_fd_0^{1/2}h_2}{D_fd_0^{1/2}h_2}=\Vert t h_2 \Vert^2.
\end{align*}
Let us see that $$\Vert th_2 \Vert^2 \geq \Vert T_0 h_2 \Vert^2.$$
In fact, $(h_2,th_2) \in t \subseteq T.$ Since $T=T_0 \ \hat{\oplus} \ (\{0\} \times \mul(T)),$ $$(h_2,th_2)=(h_2,T_0h_2)+(0,z)$$ for some $z \in \mul(T) .$
Then $th_2=T_0h_2+z.$ Since $T_0h_2 \in \ran(T_0) \subseteq \cldom(T^*)\subseteq \mul(T)^{\perp}$ and $\mc{D}_1 \subseteq \St$ it follows that
$$\Vert t h_2 \Vert^2=\Vert T_0 h_2 \Vert^2 + \Vert z \Vert^2 \geq \Vert T_0 h_2 \Vert^2=\Vert C_0^{1/2}h \Vert^2.$$ 
Then
$$\PI{A_0h}{h}=\Vert A_0^{1/2}h\Vert^2 \geq \Vert C_0^{1/2}h \Vert^2 \mbox{ for every } h \in \dom(A).$$
Since $\dom(A)$ is a core for $A_0^{1/2},$ by \cite[Lemma 10.10]{Schmudgen}, it follows that $\dom(A_0^{1/2}) \subseteq \dom(C_0^{1/2}) $ and $\Vert A_0^{1/2}h\Vert \geq \Vert C_0^{1/2}h \Vert \mbox{ for every } h \in \dom(A_0^{1/2}).$	Hence, $A\geq C.$ So that $$C \in \mc{M}(A,\St^{\perp}).$$
	
	Let $X \in \mc{M}(A,\St^{\perp}).$ Then, by Lemma \ref{lemalr}, there exists a contraction $W \in L(\HH)$ such that
	$$X_0^{1/2}\supset WA_0^{1/2},$$ where $X_0$ is the operator part of $X.$ Recall that $X_0$ is a nonnegative selfadjoint linear operator in $\cldom(X).$
	Also, if $h_2 \in \mc{D}_2 \subseteq \dom(A) =\dom(A_0) \subseteq \dom(A_0^{1/2}),$ 
	$$X_0^{1/2}h_2=WA_0^{1/2}h_2=WV_2d_0^{1/2}=W'd_0^{1/2}h_2,$$ with $W'=WV_2.$  
	Also, since $X \leq A,$ we have that $\dom(A) \subseteq \dom(A_0^{1/2}) \subseteq \dom(X_0^{1/2})$ and 
	$$\PI{X_0^{1/2}h}{X_0^{1/2}h} \leq \PI{A_0^{1/2}h}{A_0^{1/2}h}= \PI{A_0h}{h} \mbox{ for every } h \in \dom(A).$$
	Let $h=\vecd{h_1}{h_2} \in \mc{D}_1 \oplus \mc{D}_2.$ Then, since $\mc{D}_1 \subseteq \St \subseteq \ker(X) =\ker(X_0),$ 
	\begin{align*}
		\PI{X_0^{1/2}h}{X_0^{1/2}h}&=\PI{X_0^{1/2}h_2}{X_0^{1/2}h_2}=\PI{W'd_0^{1/2}h_2}{W'd_0^{1/2}h_2}\\
		&=\PI{\matriz{0}{0}{0}{W'^*W'}\vecd{0}{d_0^{1/2}h_2}}{\vecd{0}{d_0^{1/2}h_2}}\\
		&=\PI{\matriz{0}{0}{0}{W'^*W'}\vecd{a_0^{1/2}h_1}{d_0^{1/2}h_2}}{\vecd{a_0^{1/2}h_1}{d_0^{1/2}h_2}}\\
		& \leq \PI{A_0h}{h}=\PI{\matriz{1}{f}{f^*}{1}\vecd{a_0^{1/2}h_1}{d_0^{1/2}h_2}}{\vecd{a_0^{1/2}h_1}{d_0^{1/2}h_2}}.
	\end{align*}

	Since $\mc{D}_1$ is a core of $a_0^{1/2}$ and  $\mc{D}_2$ is a core of $d_0^{1/2},$ we have that $\ol{a_0^{1/2}(\mc{D}_1)}=\clran{(a_0^{1/2})}$ and $\ol{d_0^{1/2}(\mc{D}_2)}=\clran{(d_0^{1/2})}.$ Also, $\ker(d_0^{1/2})=\ker(V_2) \subseteq \ker(W') \cap \ker(f)$ and $\ker(a_0^{1/2}) \subseteq \ker(f^*).$ Hence, by the last inequality, it follows that
	$$0 \leq \matriz{0}{0}{0}{W'^*W'} \leq \matriz{1}{f}{f^*}{1},$$ where the inequality holds in the Hilbert space $\cldom(A)=\N_1 \oplus \N_2.$ 
	Therefore
	$$\matriz{0}{0}{0}{W'^*W'} \leq \matriz{1}{f}{f^*}{1}_{/ \N_1}=\matriz{0}{0}{0}{1-f^*f}.$$  So that $W'^*W' \leq 1-f^*f.$ Then
	\begin{align*}
		\PI{X_0^{1/2}h}{X_0^{1/2}h}&=\PI{W'd_0^{1/2}h_2}{W'd_0^{1/2}h_2}\\ 
		&\leq \PI{(1-f^*f)^{1/2}d_0^{1/2}h_2}{(1-f^*f)^{1/2}d_0^{1/2}h_2}\\
		&=\PI{D_fd_0^{1/2}h_2}{D_fd_0^{1/2}h_2}=\Vert D_fd_0^{1/2}h_2 \Vert^2=\Vert th_2 \Vert^2.
	\end{align*}

	Next we show  that $C \geq X.$ Let $h=\vecd{h_1}{h_2} \in \dom(C_0^{1/2}) = \St \oplus \dom(T_0).$ Then $h_2 \in \dom(T_0).$ So that there exists $k \in \N_2$ such that $(h_2,k) \in T_0 \subset T.$ Since $T_0$ is an operator, it follows that $k=T_0h_2.$
	Also,  since $(h_2,k) \in T=\ol{D_gd^{1/2}|_{\mc{D}_2}},$ there exists a sequence
	$(h_n, y_n)_{n \geq 1} \in D_gd^{1/2}|_{\mc{D}_2}$ such that $\underset{n \ra \infty}{\lim}(h_n,y_n) = (h_2,k).$ 
	
	Since $(h_n, y_n) \in D_gd^{1/2}|_{\mc{D}_2}=D_fd_0^{1/2}|_{\mc{D}_2} \ \hat{\oplus} \ (\{0\} \times \M_2)$ for every $n \in \mathbb{N},$ then $h_n \in \mc{D}_2$ and, for every $n \in \mathbb{N},$ there exits $m_n \in \M_2$ such that $$(h_n,y_n)=(h_n, D_fd_0^{1/2}h_n)+(0,m_n).$$
	Then, $\underset{n \ra \infty}{\lim}h_n =h_2$ and $\underset{n \ra \infty}{\lim} D_fd_0^{1/2}h_n+m_n =k.$ But, since $D_fd_0^{1/2}h_n \in \N_2$ for every $n \in \mathbb{N}$  and $k \in \N_2 \perp \M_2,$ it follows that $\underset{n \ra \infty}{\lim} m_n=0$ and then $\underset{n \ra \infty}{\lim} D_fd_0^{1/2}h_n=\underset{n \ra \infty}{\lim} th_n = k.$
	 From $$\Vert X_0^{1/2}h_n \Vert^2 \leq \Vert t h_n\Vert^2 \mbox{ for every } n \in \mathbb N,$$ it follows that
	 $(X_0^{1/2}h_n)_{n \geq 1}$ is a Cauchy sequence (so it converges). From the fact that $X_0^{1/2}$ is a closed operator, $h_2 \in \dom(X_0^{1/2})$  and $\underset{n \ra \infty}{\lim} X_0^{1/2}h_n =X_0^{1/2}h_2.$ Then, since $\St \subseteq \ker(X_0)=\ker(X_0^{1/2}) \subseteq \dom(X_0^{1/2}),$ $$\dom(C_0^{1/2})= \St  \oplus \dom(T_0) \subseteq \dom(X_0^{1/2}).$$
	Therefore, since $h_1 \in \ker(X_0^{1/2}),$
\begin{align*}
\Vert X_0^{1/2}h \Vert&=\Vert X_0^{1/2}h_2 \Vert=\underset{n \ra \infty}{\lim} \Vert X_0^{1/2}h_n \Vert \\
&\leq \underset{n \ra \infty}{\lim} \Vert th_n \Vert =\Vert k \Vert=\Vert T_0 h_2 \Vert=\Vert C_0^{1/2} h\Vert.
\end{align*}
\end{proof}

\begin{obs} 
	Suppose that $A \geq 0$ is (a densely defined) operator in $\HH.$ If $X \in \mc{M}(A,\St^{\perp})$ then $X$ is an operator in $\HH.$ In fact, if  $X \in \mc{M}(A,\St^{\perp})$ then $\dom(A^{1/2}) \subseteq \dom(X^{1/2})$ and then $$\mul(X) =\mul(X^{1/2}) =\dom(X^{1/2})^{\perp} \subseteq \dom(A^{1/2})^{\perp} =\mul(A)=\{0\}.$$
	In this case, $\N_1=\ol{\mc{D}_1 }=\St$ and $\M_1=\M_2=\{0\}.$ So that $f=g,$ $d=d_0,$ $T^*T=t^{\times}\ol{t}$ and, 
	$$A_{/ \St}=\matriz{0}{0}{0}{T^*T}=\max \ \{X \mbox{ l.o. in } \HH: 0 \leq X \leq A,  \ran(X) \subseteq \St^{\perp}\}.$$
\end{obs}

\bigskip
In a similar way, we now define $A_{\St}$ the \emph{compression} of $A.$ For this, consider the row linear relation 
$$S:=\vect{a^{1/2}|_{\mc{D}_1}}{gd^{1/2}|_{\mc{D}_2}} \subseteq \HH \times \St$$ with $\dom(S)= \mc{D}_1\oplus \mc{D}_2=\dom(A).$  
Define $A_{ \St}$ by $$A_{\St}:=S^*\ol{S}.$$ Then, by Theorem \ref{thmVN}, $A_{\St}$ is a nonnegative selfadjoint linear relation $\HH.$  

\begin{lema}\label{lemaRL2} Under the above hypotheses, $\ol{S}$ is decomposable and
$$A_{ \St}=s^{\times} \ol{s}\ \hat{\oplus} \ ( \{0\} \times\mul(A)),$$
where $s: \mc{D} \ra \N_1$ is the closable linear operator defined by \begin{equation} \label{sdef}
s:=\vect{a_0^{1/2}|_{\mc{D}_1}}{fd_0^{1/2}|_{\mc{D}_2}}
\end{equation} and $s^{\times}$ is the adjoint of $s$ when viewed as an operator from $\cldom(S)$ to $\cldom(S^*).$
\end{lema}

\begin{proof} Since $a^{1/2}|_{\mc{D}_1}=a_0^{1/2}|_{\mc{D}_1} \ \hat{\oplus} \ (\{0\} \times \M_1)$ and $gd^{1/2}|_{\mc{D}_2}=fd_0^{1/2}|_{\mc{D}_2},$ it follows that
$$S=s \ \hat{\oplus} \ (\{0\} \times \M_1).$$ 
In fact, it is clear that $\ran(s)\subseteq \N_1$ and, since $\mul(gd^{1/2}|_{\mc{D}_2})=\{0\},$  $\mul(S)=\mul(a^{1/2}|_{\mc{D}_1}) + \mul(gd^{1/2}|_{\mc{D}_2}) =\M_1$ and $\dom(S)=\dom(A)=\dom(s).$ Also, if $(h,y) \in S$ then $(h,y)=\left(\vecd{h_1}{h_2},y_1+y_2\right)$ where $h_1 \in \mc{D}_1,$ $h_2 \in \mc{D}_2$ and $(h_1,y_1) \in  a^{1/2}|_{\mc{D}_1}$ and $(h_2,y_2) \in gd^{1/2}|_{\mc{D}_2}=fd_0^{1/2}|_{\mc{D}_2}.$
So that $(h_1,y_1)=(h_1,a_0^{1/2}h_1)+(0,m_1)$ for some $m_1 \in \M_1$ and $y_2=fd_0^{1/2}h_2.$ Hence
\begin{align*}
(h,y)&=\left(\vecd{h_1}{h_2},y_1+y_2\right)\\
&=\left(\vecd{h_1}{h_2},a_0^{1/2}h_1+fd_0^{1/2}h_2\right)+(0,m_1) \in s \ \hat{\oplus} \ (\{0\} \times \M_1).
\end{align*}
Then $S \subset s \ \hat{\oplus} \ (\{0\} \times \M_1)$ and, by \cite[Corollary 2.2]{Hassi}, $S=s \ \hat{\oplus} \ (\{0\} \times \M_1).$

The row operator  $s$ is closable, in fact, $s^{\times}=\vecd{a_0^{1/2}}{d_0^{1/2}f^*}$ and, as $a_0^{1/2}(\mc{D}_1)\subseteq \dom(a_0^{1/2}) \cap \dom(d_0^{1/2}f^*)$ and  $\ker(a_0^{1/2}) \subseteq \dom(a_0^{1/2}) \cap \ker(f^*),$
$$\dom(s^\times)\supseteq a_0^{1/2}(\mc{D}_1) \oplus \ker(a_0^{1/2})$$ which is dense in $\N_1.$
Then $\ol{s}$ is an operator. Moreover, by Theorem \ref{TdecompHassi}, $\ol{S}$ is decomposable and
$$\ol{S}=\ol{s} \  \hat{\oplus} \ (\{0\} \times \M_1).$$

Also, since $\mc{D}_1$ is a core of $a^{1/2}$ and $\mc{D}_2$ is a core of $d^{1/2},$ it follows that
$$S^*=\vecd{a^{1/2}}{d^{1/2}g^*},$$
$\mul(A_{\St})=\mul(S^*)=\mul(a^{1/2}) \oplus \mul(d^{1/2}g^*)=\M_1 \oplus \M_2=\mul(A)$ and, by Theorem \ref{thmVN},  the operator part of $ S^*\ol{S}$ is 
$(S^*\ol{S})_0=((\ol{S})_0)^\times (\ol{S})_0=s^\times\ol{s}.$ 
Then $$A_{ \St}=s^\times \ol{s} \ \hat{\oplus } \ (\{0\} \times \mul(A)).$$
\end{proof}

Let $V_1$ be the partial isometry given in the proof of Theorem \ref{thmRMRL}.
Then 
\begin{equation} \label{eqS}
s={V_1}^*A_0^{1/2} \mbox{ on } \dom(A).
\end{equation}

\begin{prop} Let $A \geq 0$ be a linear relation in $\HH$ and let $\St$ be a closed subspace of $\HH$ such that $P_{\St} (\dom(A)) \subseteq \dom(A).$ Then
	$$A \geq A_{\St}.$$ 
\end{prop}	
\begin{proof}  Suppose that $(A_{\St})_0$ is the operator part of $A_{\St}$  then, by \cite[Proposition 2.7]{Hassi}, 
$$\PI{(A_{\St})_0^{1/2}u}{(A_{\St})_0^{1/2}v}=\PI{(\ol{S})_0u}{(\ol{S})_0v}=\PI{\ol{s}u}{\ol{s}v},$$ 
for every  $u, v \in \dom((A_{\St})_0^{1/2})=\dom((\ol{S})_0)=\dom(\ol{s}).$ Then
	\begin{align*}
		\dom(A)=\dom(s) \subseteq \dom((A_{\St})_0^{1/2}).
	\end{align*}
 Let $h=\vecd{h_1}{h_2} \in \mc{D}_1 \oplus \mc{D}_2=\dom(s).$ Then, by \eqref{eqS},
	\begin{align*} \Vert (A_{\St})_0^{1/2} h \Vert&=\Vert\ol{s} h \Vert=\Vert sh \Vert =\Vert V_{1}^*A_0^{1/2} h \Vert \leq \Vert A_0^{1/2} h \Vert=\Vert A_0^{1/2}h \Vert.
	\end{align*}	
	Hence, since $\dom(A)$ is a core of $A^{1/2},$ by \cite[Lemma 10.10]{Schmudgen}, $A \geq A_{\St}.$
\end{proof}

Define $$\LL:=\ol{A^{1/2}(\mc{D}_1)} \cap \cldom(A).$$
In the following we show that if the positive relations $A$ and $A^{1/2}$ admit a matrix representation with respect to $\St \oplus \St^{\perp}$ and $\LL \oplus \LL^{\perp},$ respectively, then $$A=A_{\St} + A_{ / \St}.$$

\begin{lema} \label{Propcloslr}  Let $A \geq 0$ be a linear relation in $\HH$ and let $\St$ be a closed subspace of $\HH$ such that $P_{\St} (\dom(A)) \subseteq \dom(A).$ Consider the matrix representation of $A$ with respect to $\St \oplus \St^{\perp}$ in \eqref{repreAS}. Then the following are equivalent:
	\begin{enumerate}
		\item $P_{\LL} (A^{1/2}(\dom(A)) \subseteq \dom(A^{1/2});$
		\item  $\dom(d^{1/2}g^*gd^{1/2}|_{\mc{D}_2})=\mc{D}_2;$ 
		\item  $\dom(d^{1/2}D_g^2d^{1/2}|_{\mc{D}_2})=\mc{D}_2.$ 
	\end{enumerate}
In this case, the linear relation $D_gd^{1/2}$ is decomposable.
\end{lema}

\begin{proof} Since $A^{1/2}(\dom(A)) =A_0^{1/2}(\dom(A)) \oplus \mul(A)$ and $\mul(A) \subseteq \LL^{\perp},$ it follows that
\begin{equation} \label{eqif}
P_{\LL} (A^{1/2}(\dom(A))) = P_{\LL} (A_0^{1/2}(\dom(A)) \oplus \mul(A)) = P_{\LL} (A_0^{1/2}(\dom(A))).
\end{equation}

Let $V_1$ and $V_2$ be the partial isometries given in  the proof of Theorem \ref{thmRMRL}.  
Then $f={V_1}^* V_2$ and, since $\LL=\ol{A_0^{1/2}(\mc{D}_1)}$, $P_{\LL}=V_1{V_1}^*.$ Also, 
$$A_0^{1/2}|_{\dom(A)}=\vect{V_1a_0^{1/2}|_{\mc{D}_1}}{V_2d_0^{1/2}|_{\mc{D}_2}},$$ and
$A_0^{1/2}=\vecd{a_0^{1/2}{V_1}^*}{d_0^{1/2}V_2^*},$ so that
$\dom(A_0^{1/2})=\dom(a_0^{1/2}{V_1}^*) \cap\dom(d_0^{1/2}V_2^*).$
Then 
\begin{equation} \label{eqif2}
P_{\LL} (A_0^{1/2}({\mc{D}_2})) \subseteq \dom(A_0^{1/2}) \Leftrightarrow
\dom(d_0^{1/2}f^*fd_0^{1/2}|_{\mc{D}_2})=\mc{D}_2.
\end{equation}
In fact, by Theorem \ref{thmRMRL},
$$
V_{1}^*P_{\LL} (A_0^{1/2}({\mc{D}_2}))=fd_0^{1/2}(\mc{D}_2) \subseteq \dom(a_0^{1/2})
$$
and
$$
V_{2}^*P_{\LL}  (A_0^{1/2}({\mc{D}_2}))=f^*fd_0^{1/2}(\mc{D}_2).
$$
Then \eqref{eqif2} follows. 

Since $gd^{1/2}|_{\mc{D}_2}=fd_0^{1/2}|_{\mc{D}_2}$ we have that \begin{equation}\label{dis1}
d^{1/2}g^*gd^{1/2}|_{\mc{D}_2}=d_0^{1/2}f^* fd_0^{1/2}|_{\mc{D}_2} \ \hat{\oplus} \ (\{0\} \times \M_2).
\end{equation}
Then  $(\mathit{i}\kern0.5pt) \Leftrightarrow (\mathit{ii}\kern0.5pt)$ follows from \eqref{eqif} and \eqref{eqif2}.

Applying \eqref{eqDg}, it can be seen that
\begin{equation}\label{dis2}
 d^{1/2}D_g^2d^{1/2}|_{\mc{D}_2}=d_0^{1/2}D_f^2d_0^{1/2}|_{\mc{D}_2} \ \hat{\oplus} \ (\{0\} \times \M_2).
\end{equation}
By \eqref{dis1},  $\dom(d^{1/2}g^*gd^{1/2}|_{\mc{D}_2})=\mc{D}_2$ if and only if $\dom(d_0^{1/2}f^* fd_0^{1/2}|_{\mc{D}_2})=\mc{D}_2.$
Then $(\mathit{ii}\kern0.5pt) \Leftrightarrow (\mathit{iii}\kern0.5pt)$ follows from \eqref{dis2} and from the fact that  $f^*fd_0^{1/2}|_{\mc{D}_2}+D_f^2d_0^{1/2}|_{\mc{D}_2}=d_0^{1/2}|_{\mc{D}_2}.$ 

Since equation \eqref{eqDg} holds and $\mul(D_gd^{1/2}|_{\mc{D}_2})=\M_2$ to see that $D_gd^{1/2}$ is decomposable it is sufficient to prove that the operator $D_fd_0^{1/2}$ is closable \cite[Theorem 3.10]{Hassi2}.
In fact, let $(y_n)_{n \geq 1} \subseteq \mc{D}_2$ be such that $y_n \ra 0$ and $D_fd_0^{1/2}y_n \ra h.$ Then, for every $h_2 \in \mc{D}_2,$
\begin{align*}
	\PI{h}{D_fd_0^{1/2}h_2}&=\underset{n \ra \infty}{\lim}\PI{D_fd_0^{1/2}y_n}{D_fd_0^{1/2}h_2}\\
	&=\underset{n \ra \infty}{\lim}\PI{y_n}{d_0^{1/2}D_f^2d_0^{1/2}h_2}=0,
\end{align*}
where we used that, by \eqref{eqif2}, $\dom(d_0^{1/2}D_f^2d_0^{1/2}|_{\mc{D}_2})=\mc{D}_2.$ Then $h \in \clran{(D_fd_0^{1/2})} \cap \ran(D_fd_0^{1/2})^{\perp}$ and  $h=0.$ 
\end{proof}

\begin{thm} \label{propsuma} Let $A \geq 0$ be a linear relation in $\HH,$ let $\St$ be a closed subspace of $\HH$ such that $P_{\St} (\dom(A)) \subseteq \dom(A).$
Then the following are equivalent:
\begin{enumerate}
	\item $\dom(A) \subseteq \dom(A_{\St});$
	\item $P_{\LL}( A^{1/2}({\dom(A)})) \subseteq \dom(A^{1/2});$
	\item $A=A_{\St}+A_{/ \St}.$
\end{enumerate}
\end{thm}	

\begin{proof}  $(\mathit{i}\kern0.5pt) \Rightarrow (\mathit{ii}\kern0.5pt)$: Let us see that $\dom(d_0^{1/2}f^*fd_0^{1/2}|_{\mc{D}_2})=\mc{D}_2.$
In fact, let $h_2 \in \mc{D}_2$ then $h_2\in  \dom(A_\St)=\dom(s^{\times}\ol{s}),$  where $s$ is as in \eqref{sdef}, and $s^{\times}\ol{s}$ is the operator part of $A_\St.$ Since $h_2 \in \mc{D}_2 \subseteq \dom(s)$ and $s$ is closable, it follows that 
$$\ol{s}h_2=sh_2=fd_0^{1/2}h_2 \in \dom(s^\times)=\dom(a_0^{1/2}) \cap \dom(d_0^{1/2}f^*).$$  Hence  $h_2 \in 
\dom(d_0^{1/2}f^*fd_0^{1/2}|_{\mc{D}_2}).$ 
Then, by \eqref{eqif} and \eqref{eqif2}, $P_{\LL}( A^{1/2}({\dom(A)})) \subseteq \dom(A^{1/2}).$
	
$(\mathit{ii}\kern0.5pt) \Rightarrow (\mathit{iii}\kern0.5pt)$: By the proof of Lemma \ref{Propcloslr}, $$\dom(d_0^{1/2}f^*fd_0^{1/2}|_{\mc{D}_2})=\dom(d_0^{1/2}D_f^2d_0^{1/2}|_{\mc{D}_2})=\mc{D}_2.$$ 
Also, since  $$g^*gd^{1/2}|_{\mc{D}_2}+D_g^2d^{1/2}|_{\mc{D}_2}=d^{1/2}|_{\mc{D}_2}$$ and
$\dom(d^{1/2}g^*gd^{1/2}|_{\mc{D}_2})=\dom(d^{1/2}D_g^2d^{1/2}|_{\mc{D}_2})=\mc{D}_2 \subseteq \dom(d^{1/2})$ (see Lemma \ref{Propcloslr}),  it follows that
\begin{equation} \label{dg}
d^{1/2}g^*gd^{1/2}|_{\mc{D}_2}+d^{1/2}D_g^2d^{1/2}|_{\mc{D}_2}=d|_{\mc{D}_2}.
\end{equation}
Next we show that $$s^{\times}s= \matriz{a_0}{b_0}{c_0}{d_0^{1/2}f^*fd_0^{1/2}|_{\mc{D}_2}}.$$ 
Let $h=\vecd{h_1}{h_2} \in \mc{D}_1 \oplus \mc{D}_2.$ Then 
$$s^{\times}sh=\vecd{a_0^{1/2}(a_0^{1/2}h_1+fd_0^{1/2}h_2)}{d_0^{1/2}f^*(a_0^{1/2}h_1+fd_0^{1/2}h_2)}=\matriz{a_0}{b_0}{c_0}{d_0^{1/2}f^*fd_0^{1/2}|_{\mc{D}_2}}h,$$ where the last equality follows from the fact that, since  $fd_0^{1/2}h_2 \in \dom(d_0^{1/2}f^*),$ it is possible to distribute.
Then, since $d^{1/2}g^*gd^{1/2}|_{\mc{D}_2}=d_0^{1/2}f^*fd_0^{1/2}|_{\mc{D}_2}  \ \hat{\oplus} \ ( \{0\} \times \M_2),$ it follows that
\begin{equation} \label{RMAS}
A_{\St} \supset s^{\times}s \ \hat{\oplus} \ ( \{0\} \times \mul(A))=\matriz{a}{b}{c}{d^{1/2}g^*gd^{1/2}|_{\mc{D}_2}}.
\end{equation}
 Clearly,  $$A_{/ \St} \supset \matriz{0}{0}{0}{d^{1/2}D_g^2 d^{1/2}|_{\mc{D}_2}}.$$ Then, by \cite[Lemma 5.5]{HassiMR} and \eqref{dg},
\begin{align} \label{A0sum}
\matriz{a}{b}{c}{d^{1/2}g^*gd^{1/2}|_{\mc{D}_2}} +\matriz{0}{0}{0}{d^{1/2}D_g^2d^{1/2}|_{\mc{D}_2}}=\matriz{a}{b}{c}{d|_{\mc{D}_2}}= A. 
\end{align}
Hence $A_{\St} + A_{/  \St} \supset A$ and, by \eqref{sum}, $$A=A^* \supset (A_{\St} + A_{/  \St})^*  \supset (A_{\St})^* + (A_{/  \St})^*= A_{\St} + A_{/  \St} \supset A.$$ So that $A=A_{\St}+A_{/ \St}.$ 

$(\mathit{iii}\kern0.5pt) \Rightarrow (\mathit{i}\kern0.5pt)$: It is straightforward.

%
\end{proof}

For a nonnegative operator $A \in L(\HH)$ and a closed subspace $\St \subseteq \HH,$ Pekarev
\cite{Pekarev} showed that the Schur complement $A_{ / \St}$ can be expressed as $A_{/\St}=A^{1/2}P_{\LL^{\perp}}A^{1/2}$ where $\LL =\ol{A^{1/2}(\St)}.$ 
In what follows, we extend this formula for a linear relation $A \geq 0$ in $\HH$ such that $P_{\St} (\dom(A)) \subseteq \dom(A)$ and $P_{\LL}( A^{1/2}({\dom(A)})) \subseteq \dom(A^{1/2}).$ 

\begin{cor} \label{proppekarev} Let $A \geq 0$ be a linear relation in $\HH,$ let $\St$ be a closed subspace of $\HH$ such that $P_{\St} (\dom(A)) \subseteq \dom(A)$ and $P_{\LL}( A^{1/2}({\dom(A)})) \subseteq \dom(A^{1/2}).$ Then
$$A_{/ \St} =A^{1/2}\ol{P_{\LL^{\perp}}A^{1/2}|_{\dom(A)}}, \ \ A_{\St}=A^{1/2}\ol{P_{\LL}A^{1/2}|_{\dom(A)}}.$$
\end{cor}

\begin{proof} Let $h=h_1+h_2 \in \mc{D}_1\oplus \mc{D}_2.$ Then 
\begin{align*}
\Vert t h_2 \Vert^2&=\PI{D_fd_0^{1/2}h_2}{D_fd_0^{1/2}h_2}=\PI{(1-f^*f)d_0^{1/2}h_2}{d_0^{1/2}h_2}\\
&=\PI{V_2^*(1-V_1{V_1}^*)V_2d_0^{1/2}h_2}{d_0^{1/2}h_2}=\PI{(1-P_\LL)A_0^{1/2}h_2}{A_0^{1/2}h_2}\\
&=\PI{P_{\LL^{\perp}}A_0^{1/2}h_2}{A_0^{1/2}h_2}=\Vert P_{\LL^{\perp}}A_0^{1/2}h\Vert^2,
\end{align*}
where we used that $P_{\LL^{\perp}}A_0^{1/2}h=P_{\LL^{\perp}}A_0^{1/2}h_2,$ because $A_0^{1/2}h_1 \in \LL.$ Then, since $t$ is closable (see Lemma \ref{Propcloslr}), $P_{\LL^{\perp}}A_0^{1/2}|_{\dom(A)}$ is also closable. 
Set $W:=\ol{P_{\LL^{\perp}}A^{1/2}|_{\dom(A)}}.$ Then, since
$	P_{\LL^{\perp}} A^{1/2}|_{\dom(A)}=P_{\LL^{\perp}}A_0^{1/2}|_{\dom(A)} \ \hat{\oplus} \ (\{0\} \times \mul(A)),$
it follows that
\begin{equation}\label{eqW}
W=\ol{P_{\LL^{\perp}}A_0^{1/2}|_{\dom(A)}}  \ \hat{\oplus} \ (\{0\} \times \mul(A)).
\end{equation} 
Moreover, since $t$ and $P_{\LL^{\perp}}A_0^{1/2}|_{\dom(A)}$ are closable operators, by \eqref{eqT} and \eqref{eqW}, it follows that the operator part of $T$ is $T_0=\ol{t}$ and the operator part of $W$ is $W_0=\ol{P_{\LL^{\perp}}A_0^{1/2}|_{\dom(A)}} .$ Also,  
$$\LL^{\perp} \cap \dom(A_0^{1/2})\subseteq A_0^{-1/2}(\St^{\perp}):=\{y \in \dom(A_0^{1/2}): A_0^{1/2}y \in \St^{\perp}\}.$$
In fact, let $y\in  \LL^{\perp} \cap \dom(A_0^{1/2}).$ Then, for every $h_1\in \mc{D}_1,$
$$0=\PI{y}{A_0^{1/2}h_1}=\PI{A_0^{1/2}y}{h_1}.$$
So that
$$A_0^{1/2}y \in \mc{D}_1^{\perp}=(\St^{\perp} \oplus \M_1) \cap \cldom(A)\subseteq \St^{\perp}$$
because $\M_1= \St \cap \mul(A).$ 
Then 
\begin{equation} \label{ran}
	\ran({W_0}^*W_0) \subseteq \St^{\perp}.
\end{equation}
In fact, let $y \in \ran({W_0}^*W_0).$ Then, since $\ran(W_0) \subseteq \LL^{\perp},$ it follows that $$y={W_0}^*W_0x=A_0^{1/2}W_0x,$$ for some $x \in \dom({W_0}^*W_0).$ Then $W_0x \in \LL^{\perp} \cap \dom(A_0^{1/2}) \subseteq A_0^{-1/2}(\St^{\perp})$ and $y=A_0^{1/2}W_0x \in \St^{\perp}.$ 
So that, by \eqref{ran}, $\St \subseteq \ker(W_0^*W_0)=\ker(W_0)\subseteq \dom(W_0),$ where we used Theorem \ref{thmVN}. 
Hence
\begin{equation} \label{eq1na}
	h \in \dom(W_0) \Leftrightarrow P_{\St^{\perp}}h \in \dom(T_0) \mbox{ and } \Vert W_0h\Vert =\Vert T_0P_{\St^{\perp}}h \Vert.
\end{equation}

Now we  show that
$$A_{/ \St}=\matriz{0}{0}{0}{T^*T}=W^*W=A^{1/2}\ol{P_{\LL^{\perp}}A^{1/2}|_{\dom(A)}},$$
where for the last equality we used that $\ran(\ol{P_{\LL^\perp}A^{1/2}|_{\dom(A)}}) \subseteq \LL^{\perp}.$
 
Suppose that $(W^*W)_0$ is the operator part of $W^*W$ then, by \cite[Proposition 2.7]{Hassi}, 
$$\PI{(W^*W)_0^{1/2}u}{(W^*W)_0^{1/2}v}=\PI{W_0u}{W_0v},$$ 
for every  $u, v \in \dom((W^*W)_0^{1/2})=\dom(W_0).$

Suppose that $(A_{/ \St})_0$ is the operator part of $A_{/ \St}.$
Let  $h \in \dom((A_{/ \St})_0^{1/2})=\St \oplus \dom(T_0)$ then $h=h_1+h_2$ with $h_1 \in \St$ and $h_2 \in \dom(T_0).$ Then, by \eqref{eq1na}, $h \in \dom(W_0).$
Conversely, if $h \in \dom(W_0),$ by \eqref{eq1na}, $P_{\St^{\perp}}h \in \dom(T_0).$ Then $h=P_{\St}h+P_{\St^{\perp}}h \in \St \oplus \dom(T_0)=\dom((A_{/ \St})_0^{1/2}).$

Also, if $h\in  \dom((W^*W)_0^{1/2})=\dom(W_0)=\dom((A_{/ \St})_0^{1/2}),$ it follows that $h=h_1+h_2 \in \St \oplus \St^{\perp}$ and, by \eqref{eq1na}, 
$$\Vert (A_{/ \St})_0^{1/2} h \Vert=\Vert T_0 h_2 \Vert = \Vert W_0 h\Vert=\Vert  (W^*W)_0^{1/2} h \Vert.$$
Then $A_{/ \St} = W^*W.$   

Finally, by \eqref{eqS},
$$V_1s=P_{\LL}A_0^{1/2}|_{\dom(A)}.$$ Then, since $s$ is closable and $V_1$ is a partial isometry, the operator $P_{\LL}A_0^{1/2}|_{\dom(A)}$ is closable and 
$$\ol{s}=\ol{V_1^*P_{\LL}A_0^{1/2}|_{\dom(A)}}={V_1}^*\ol{P_{\LL}A_0^{1/2}|_{\dom(A)}}.$$ 
So that
$$s^{\times}\ol{s}=A_0^{1/2}P_{\LL}V_1{V_1}^* \ol{P_{\LL}A_0^{1/2}|_{\dom(A)}}=A_0^{1/2}P_{\LL}\ol{P_{\LL}A_0^{1/2}|_{\dom(A)}},$$
and, since $\ran(\ol{P_{\LL}A^{1/2}|_{\dom(A)}}) \subseteq \LL,$ 
\begin{align*}
	A^{1/2}P_{\LL}\ol{P_{\LL}A^{1/2}|_{\dom(A)}}&=A^{1/2}\ol{P_{\LL}A^{1/2}|_{\dom(A)}}=A_0^{1/2}\ol{P_{\LL}A_0^{1/2}|_{\dom(A)}} \ \hat{\oplus} \ ( \{0\} \times \mul(A))\\
	&=s^{\times}\ol{s} \ \hat{\oplus} \ ( \{0\} \times \mul(A))=A_{\St}.
\end{align*}

\end{proof}

\begin{cor} \label{corsuma} Let $A \geq 0$ be a linear relation in $\HH$ and let $\St$ be a closed subspace of $\HH.$ 
If $A$ and $A^{1/2}$ admit a matrix representation with respect to $\St \oplus \St^{\perp}$ and $\LL \oplus \LL^{\perp},$ respectively, then $$A_{/ \St} =A^{1/2}\ol{P_{\LL^{\perp}}A^{1/2}|_{\dom(A)}}, \ \ A_{\St}=A^{1/2}\ol{P_{\LL}A^{1/2}|_{\dom(A)}}, \ \mbox{ and } A=A_{\St} + A_{ / \St}.$$
\end{cor}

\begin{proof}  By Theorem \ref{theoMLR}, $P_{\St} (\dom(A)) \subseteq \dom(A)$ and $P_{\LL} (\dom(A^{1/2})) \subseteq \dom(A^{1/2}).$ Then, since $A^{1/2}(\dom(A)) \subseteq \dom(A^{1/2})  \oplus \mul(A),$ it follows that $$P_{\LL} (A^{1/2}(\dom(A)))\subseteq P_{\LL} (\dom(A^{1/2})) \subseteq \dom(A^{1/2}).$$ Then, the result follows from Corollary \ref{proppekarev} and  Theorem \ref{propsuma}.
\end{proof}

\section*{Acknowledgments}

  Maximiliano Contino and Alejandra Maestripieri  were
  supported by CONICET PIP 0168.  The work of Stefania Marcantognini
  was done during her stay at the Instituto Argentino de Matem\'atica
  with an appointment funded by the CONICET.  She is grateful
  to the institute for its hospitality and to the CONICET for
  financing her post.

 \goodbreak

%

\end{document}